\documentclass{article}
\title{Discrete and continuous Muttalib--Borodin processes I: the hard edge}
\author{Dan Betea\thanks{Department of Mathematics, KU Leuven,  Celestijnenlaan 200b -- box 2400,
B-3001 Leuven, Belgium. The author was supported by FWO Flanders project EOS 30889451. E-mail: {\tt dan.betea@gmail.com}}
\and Alessandra Occelli\thanks{Department of Mathematics, Instituto Superior T\'ecnico, Av.~Rovisco Pais 1, 1049-001 Lisbon, Portugal. The author was supported by the HyLEF ERC starting grant 2016. E-mail: {\tt alessandra.occelli@tecnico.ulisboa.pt }}}
\date{}

\usepackage{amsmath, amsthm, amssymb, verbatim, amsfonts, bbm, mathtools}
\usepackage{color, graphicx, hyperref}
\usepackage[margin=0.9in]{geometry}

\def\Z{\mathbb{Z}}
\def\N{\mathbb{N}}
\def\R{\mathbb{R}}
\def\P{\mathbb{P}}

\def\x{\mathbf{x}}
\def\y{\mathbf{y}}

\def\u{\mathbf{u}}
\def\v{\mathbf{v}}
\def\a{\mathbf{a}}
\def\b{\mathbf{b}}
\def\Id{\mathbbm{1}}

\DeclareMathAlphabet{\mathpzc}{OT1}{pzc}{m}{it}

\newtheorem{prop}{Proposition}[section]
\newtheorem{thm}[prop]{Theorem}

\newtheorem{cla}[prop]{Claim}

\newtheorem{remark}[prop]{Remark}

\newenvironment{rem}{\begin{remark}\normalfont}{\end{remark}}

\numberwithin{equation}{section}

\allowdisplaybreaks

\begin{document}

\maketitle

\abstract{In this note we study a natural measure on plane partitions giving rise to a certain discrete-time Muttalib--Borodin process (MBP): each time-slice is a discrete version of a Muttalib--Borodin ensemble (MBE). The process is determinantal with explicit time-dependent correlation kernel. Moreover, in the $q \to 1$ limit, it converges to a continuous Jacobi-like MBP with Muttalib--Borodin marginals supported on the unit interval. This continuous process is also determinantal with explicit correlation kernel. We study its hard-edge scaling limit (around 0) to obtain a discrete-time-dependent generalization of the classical continuous Bessel kernel of random matrix theory (and, in fact, of the Meijer $G$-kernel as well). We lastly discuss two related applications: random sampling from such processes, and their interpretations as models of directed last passage percolation (LPP). In doing so, we introduce a corner growth model naturally associated to Jacobi processes, a version of which is the ``usual'' corner growth of Forrester--Rains in logarithmic coordinates. The aforementioned hard edge limits for our MBPs lead to interesting asymptotics for these LPP models. In particular, a special cases of our LPP asymptotics give rise (via the random matrix Bessel kernel and following Johansson's lead) to an extremal statistics distribution interpolating between the Tracy--Widom GUE and the Gumbel distributions.}

\tableofcontents

\section{Introduction}

\paragraph{Background.}

Muttalib--Borodin ensembles (MBEs for short) are probability measures on $n$ points $0 < x_1 < \dots < x_n$ of the from
\begin{equation}
    \P(x_1 \in d x_1, \dots, x_n \in d x_n) = Z^{-1} \prod_{1 \leq i<j \leq n} (x_j-x_i)(x_j^\theta-x_i^\theta) \prod_{i=1}^n e^{-V(x_i)}
\end{equation}
where $\theta > 0$, $V$ is a potential and $Z$ is the normalization constant (partition function). These ensembles were introduced by Muttalib~\cite{mut95} as generalizations of random matrix ensembles that are at the same time toy models for disordered conductors. He observed that these ensembles are determinantal and hoped that further analysis could be carried out in some (so-called integrable) cases. Borodin achieved this in~\cite{bor99}, explicitly computing the correlation kernels when the weight $w(x) = e^{-V(x)}$ is the (degenerate) Jacobi weight, the Laguerre and the (generalized) Hermite weight. In his most general case when $w(x) = x^\alpha \Id_{[0, 1]}$ he further proved (a somewhat weaker version of) the following hard-edge (scaling around 0) limit result:
\begin{equation}
    \lim_{n \to \infty} \P\left( \frac{x_1}{n^{1+\frac{1}{\theta}}} < s\right) = \det (1-K_B)_{L^2(0,s)}.
\end{equation} 
Here the right-hand side is a Fredhoolm determinant for the integral operator $K_B$ with kernel 
\begin{equation}\label{eq:bor_ker}
    K_B(x, y) = \theta \int_0^1 J_{\frac{\alpha+1}{\theta}, \frac{1}{\theta}}(xt) J_{\alpha+1, \theta} ((yt)^\theta) (xt)^{\alpha} dt
\end{equation}
(``B'' for Borodin) where $J$ is Wright's function~\cite[Vol.~3 Ch.~18 eq.~(27)]{bat}
\begin{equation} \label{eq:J_fn}
    J_{a, b}(x) = \sum_{k = 0}^\infty \frac{(-x)^k}{k! \Gamma(a+bk)}.
\end{equation}

In this paper, starting from a discrete model on plane partitions, we introduce (space-) discrete and continuous Muttalib--Borodin processes (MBPs). There are discrete extended-time versions of MBEs: each slice is an appropriate MBE. The continuous MBPs have as weights (mild) generalizations of Borodin's Jacobi weight. They are the space-continuous limit of a discrete MBP coming from the study of plane partitions and already briefly introduced, mildly disguised, in~\cite[Section 2.4]{fr05}. We show that both the discrete and continuous MBPs are determinantal with explicit correlation kernels. We further compute the hard-edge limit for both and obtain a discrete time extension of Borodin's kernel. Finally, we discuss natural interpretations of the hard-edge last particle position as a certain directed last-passage percolation (LPP) time in an inhomogeneous environment. At the finite pre-limit level some (but not all) of these connections are classical~\cite{fr05}. The asymptotics of these LPP times nevertheless becomes interesting in the scaling limit. Special cases have asymptotic distributions which interpolate between the classical Tracy--Widom GUE distribution~\cite{tw94_airy} (found at the soft edge of correlated systems) and the Gumbel distribution (the edge/maximum of iid random variables). The interpolation is provided by the continuous Bessel kernel of random matrix theory~\cite{tw94_bessel, for93}, and in this regard our results are similar to Johansson's~\cite{joh08}.

\paragraph{Related work.} Before we state our main results, let us comment on works relevant to ours. In this paper we will utilize the tool of (principally specialized) Schur processes~\cite{or03} and measures~\cite{oko01} as our starting point. Forrester--Rains~\cite[Section 2.4]{fr05} already mention our discrete MBEs below (though they do not consider the time-extended MBPs) coming from principally specialized Schur measures, but they do not analyze these measures any further. Borodin--Gorin--Strahov~\cite{bgs19} also use principally specialized Schur processes, but have a different goal than ours: providing overarching combinatorial interpretations of matrix corner product processes. 

On the analytical side, Kuijlaars and Molag~\cite{km19, mol20} have shown Borodin's hard-edge kernel~\ref{eq:bor_ker} is universal for a wide range of potentials, at least when $\theta^{-1} \in \Z$. The same kernel and its variants appear extensively in random matrix literature (and sometimes in combinatorics) under the name \emph{Meijer $G$-kernel}. See~\cite{aik13,kz14,bgs14,bgs19} and references therein. Our kernel of Theorem~\ref{thm:he_kernel} and~\ref{prop:fox_h} seems to generalize (a simple version of) the Meijer $G$-kernel by replacing the Meijer $G$-function with the Fox $H$-function---see~\cite{fox61} for both functions and for some Fourier analytic context where they appear. Finally, our Theorem~\ref{thm:cont_lpp} is an extension of Johansson's~\cite[Thm.~1.1a]{joh08}.

\paragraph{Main contributions.} Our main contributions are threefold (and a half):
\begin{itemize}
    \item we introduce, via plane partitions, discrete Muttalib--Borodin processes (MBPs) which are extended-time versions of the similar ensembles already introduced by Forrester--Rains~\cite{fr05}, and we show that they are determinantal point processes with explicit correlation kernels. This is contained in Theorems~\ref{thm:disc_mb_dist} and~\ref{thm:disc_corr};
    \item in the $q \to 1-$ limit of the above (where $q$ is a natural parameter), we introduce continuous MBPs which are time-extensions of the ensembles of Borodin~\cite{bor99} and~\cite{fr05}. We show these too are determinantal with explicit kernels. Furthermore, we show that at the hard edge of the support the kernel converges to the \emph{Fox $H$-kernel}, a generalization of the random matrix hard-edge kernels of Bessel~\cite{tw94_bessel} and Meijer $G$-type\footnote{For the appearance of the Meijer $G$-kernel using our discrete technique and limit, see~\cite{bgs19} where the focus is on so-called ``corner processes'' in products of random matrices.}. See Theorems~\ref{thm:cont_mb_dist}, \ref{thm:cont_corr}, and~\ref{thm:he_kernel};
    \item we naturally connect the hard-edge limits obtained to two natural models of last passage percolation with inhomogeneous weights: one in an infinite geometry where the weights decay rapidly, and one in finite geometry (which then becomes infinite in the limit). In both cases we study fluctuations of the last passage time; in the first case we see that the asymptotic distribution interpolates~\cite{joh08} between two ``classical'' extreme statistics distributions: Gumbel and Tracy--Widom GUE~\cite{tw94_airy}. These are Theorems~\ref{thm:disc_lpp} and~\ref{thm:cont_lpp}. See also Remark~\ref{rem:interpolation} for the interpolation property. As a byproduct of our interest in last passage percolation, we also discuss (existing) and introduce (new) exact sampling algorithms for the discrete and continuous MBPs we study. Informally speaking, they are variations on the Robinson--Schensted--Knuth correspondence~\cite{knu70} as reinterpreted by Fomin~\cite{fom86, fom95}.
\end{itemize}
Finally, there is a secondary and less quantifiable contribution we make with this note. We attempt to give as many equivalent formulas (and statements) for the encountered objects (and results) as is possible and feasible, perhaps out of a romantic and (mis)guided belief in the unity of mathematics. 

\paragraph{Outline.} In Section~\ref{sec:main_results} we introduce the models under study and state the main results, split along the lines described above: discrete results, continuous results, and last passage percolation results. The discrete results are proven in Section~\ref{sec:disc_proofs}, the continuous ones in Section~\ref{sec:cont_proofs}, and the last passage percolation results have proofs in Section~\ref{sec:lpp_proofs}  which also contains (and starts with) the random sampling algorithms discussed above. We conclude in Section~\ref{sec:conclusion}. In Appendix~\ref{sec:appendix} we list alternative formulas for the relevant correlation kernels we discuss; we choose this route to minimize the technicalities of Section~\ref{sec:main_results}.

\paragraph{Notations.} Many of our formulas depend on whether a number is positive or not. We will use the following standard notation to denote \emph{positive and negative parts} of a real (in our cases always integer) number $s$:
\begin{equation} \label{eq:plus_minus_parts}
        s^+ = \max(0, s), \quad s^- = (-s)^+ = \max(-s, 0)
\end{equation}
so that $s = s^+ - s^-$. We nonetheless write, whenever feasible, a more detailed version of the formula in which we employ this notation, as it can get confusing. 

The \emph{Pochhammer and $q$-Pochhammer symbols of length} $n$ and $m$ respectively are defined as:
\begin{equation}
    (x)_n = \prod_{0 \leq i < n} (x+i) = \frac{\Gamma(x+n)}{\Gamma(x)}, \quad (x; q)_m = \prod_{0 \leq i < m} (1-xq^i)
\end{equation}
where $q$ is a number (usually in $[0,1)$), $\Gamma$ is the Euler gamma function, $n \in \N$ and $m \in \N \cup \{\infty\}$ (if $n, m$=0, both equal 1). Notice $(x; q)_m = (x; q)_\infty / (x q^m; q)_{\infty}$ for $m$ finite.

\paragraph{Acknowledgements.} The authors would like to thank A. Borodin, P. Ferrari, A. Kuijlaars, L. Molag, E. Strahov and H. Walsh for fruitful conversations regarding this article.

\section{Main results} \label{sec:main_results}

\subsection{Plane partitions and discrete Muttalib--Borodin processes} \label{sec:pp}

We use the language of (plane and ordinary) partitions to state our first set of discrete results. 

A \emph{partition} $\lambda = (\lambda_1 \geq \lambda_2 \geq \dots \geq 0)$ is a sequence of non-increasing non-negative integers which is eventually zero. The non-zero entries $\lambda_i > 0$ are called \emph{parts} and their number is the \emph{length} denoted $\ell(\lambda)$. The \emph{size} of the partition $\lambda$, denoted $|\lambda|$, is the sum of all its parts $|\lambda| = \sum_i \lambda_i$. The \emph{empty partition} of size 0 is denoted by $\emptyset$. Two partitions are called \emph{interlacing} and we write $\mu \prec \lambda$ if
\begin{equation}
    \lambda_1 \geq \mu_1 \geq \lambda_2 \geq \mu_2 \geq \dots
\end{equation}

A \emph{plane partition} $\Lambda$ with base in an $M \times N$ rectangle for $1 \leq M, N \in \N$ is an array $\Lambda = (\Lambda_{i,j})_{1 \leq i \leq M, 1 \leq j \leq N}$ of non-negative integers satisfying $\Lambda_{i, j+1} \geq \Lambda_{i, j}$ and $\Lambda_{i+1, j} \geq \Lambda_{i, j}$ for all appropriate $i, j$. It can be viewed in 3D as a pile of cubes atop an $M \times N$ floor of a room (rectangle) where we place $\Lambda_{i, j}$ cubes above integer lattice point $(i, j)$ (starting from the ``back corner'' of the room). See Figure~\ref{fig:pp_mb} for an example. 

Let us fix real parameters $0 \leq a, q < 1$ and $\eta, \theta \geq 0$ \footnote{These restrictions can be somewhat relaxed though with little benefit for our exposition, so we will not do this here} and positive integers $M, N$. Without loss of generality we fix throughout
\begin{equation}
    M \leq N.
\end{equation}

\begin{figure} [!t]
    \begin{center}
        \includegraphics[scale=0.4]{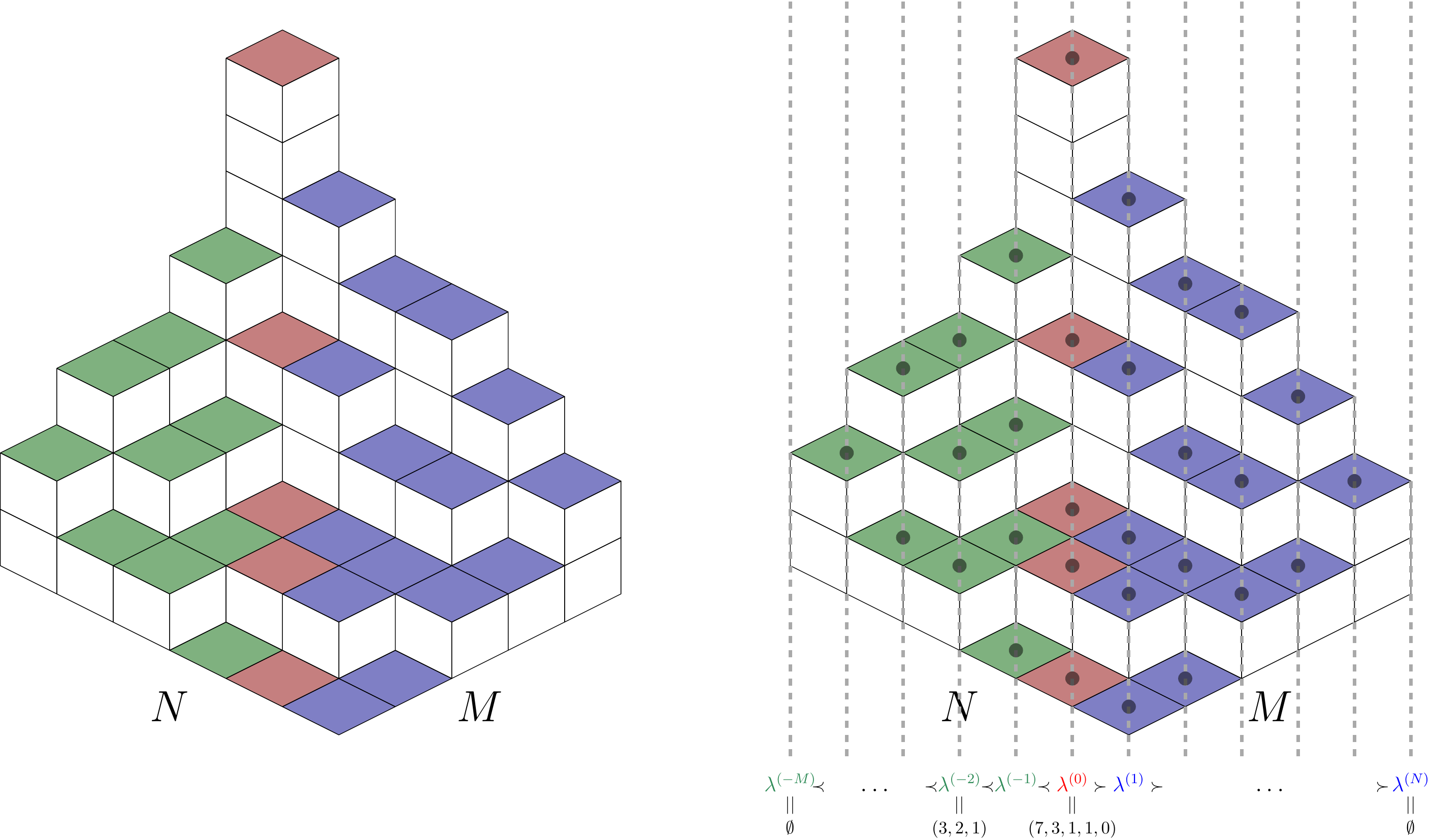}
    \end{center}
    \caption{Left: a plane partition $\Lambda$ with base in an $M \times N$ rectangle for $(M,N)=(5, 6)$. The columns of cubes contributing to the left volume (=19), central volume (=12) and right volume (=30) in equation~\eqref{eq:pp_measure} have lids shaded in different colors. Right: the sequence of interlacing partitions and, up to a shift, the points of the point process associated to $\Lambda$. We have $\text{left vol} = \sum_{i=-M}^{-1} |\lambda^{(i)}| = 19, \text{central vol} = |\lambda^{(0)}| = 12, \text{right vol} = \sum_{i=1}^{N} |\lambda^{(i)}| = 30$.}
    \label{fig:pp_mb}
\end{figure}

Consider the plane partition $\Lambda$ in Figure~\ref{fig:pp_mb} with base in an $M \times N$ rectangle. We call the \emph{central volume} (the word \emph{trace} is more customary in the literature) the total number of cubes on the central slice (marked in red). The number of cubes strictly to the right of the central slice is the \emph{right volume}, the cubes to the left give rise to the \emph{left volume}. The measure we study on such objects is 
\begin{equation} \label{eq:pp_measure}
    \P (\lambda) \propto q^{\eta \text{ left vol}} \left( a q^{\frac{\eta + \theta}{2}} \right)^{ \text{ central vol}} q^{\theta \text{ right vol}}.
\end{equation}
Plane partitions $\Lambda = (\Lambda_{i,j})_{1 \leq i \leq M, 1 \leq j \leq N}$ with base in an $M \times N$ rectangle can be seen as a sequence of $M+N+1$ interlacing regular integer partitions
\begin{equation} \label{eq:disc_interlacing}
        \Lambda = (\emptyset = \lambda^{(0)} \prec \lambda^{(-M+1)} \prec \cdots \prec \lambda^{(0)} \succ \cdots \succ \lambda^{(N-1)} \succ \emptyset = \lambda^{(N)}) 
\end{equation}
via 
\begin{equation}
    \lambda^{(t)} = \begin{dcases}
        (\Lambda_{k+|t|, k})_{k \geq 1} &\text{if }t \leq 0, \\
        (\Lambda_{k, k+t})_{k \geq 1} &\text{if } t > 0
    \end{dcases}
\end{equation} (this is depicted in Figure~\ref{fig:pp_mb}; each partition is a ``vertical slice'' of $\Lambda$ with parts given by the non-zero heights of the horizontal lozenges on that slice). We think of $t$ as discrete time. The interlacing constraints dictate that the partition at time $t$ has length at most $L_t$:
\begin{equation} \label{eq:def_Lt}
    \ell(\lambda^{(t)}) \leq L_t, \quad \text{with } L_t = \begin{cases}
        M-|t| &\text{if } -M \leq t \leq 0,\\
        \min(N-t, M) &\text{if } 0 < t \leq N.
    \end{cases}
\end{equation}
The sequence of partitions $\Lambda$ gives rise to a point process on $\{ -M, \dots, -1, 0, 1, \dots, N \} \times \N$ with exactly $L_t$ points at time $-M \leq t \leq N$ (this means no points at the extremities). The distinct particle positions at time $t$, denoted by $l^{(t)}$, are obtained by a deterministic shift~\footnote{We could shift by other amounts, notably $N$ or $L_t$ or really any integer big enough.}:
\begin{equation}
    l^{(t)}_i = \lambda^{(t)}_i + M - i, \quad 1 \leq i \leq L_t.
\end{equation}
We note that the ensemble $(l^{(t)})_{t}$ is, up to deterministic shift, just the set of positions of all horizontal lozenges in the plane partition picture (see Figure~\ref{fig:pp_mb} (right)).

Our first result\footnote{in a precise sense a time-extension of the introductory remarks of~\cite[Section 2.4]{fr05}} is that under~\eqref{eq:pp_measure}, each $l^{(t)}$ from the induced point process $(l^{(t)})_{-M+1 \leq t \leq N-1}$ has a discrete Muttalib--Borodin marginal distribution. To state it, let us make the following notation:
\begin{equation}
    Q = q^{\eta}, \quad \tilde{Q} = q^{\theta}.
\end{equation}

\begin{thm} \label{thm:disc_mb_dist}
    Under the measure~\eqref{eq:pp_measure}, for each $-M+1 \leq t \leq N-1$, each ensemble/slice $l^{(t)}$ of $L_t$ points from the process $(l^{(t)})_{-M+1 \leq t \leq N-1}$ has the following discrete Muttalib--Borodin marginal distribution:
    \begin{equation}
        \P(l^{(t)} = l) \propto \prod_{1 \leq i < j  \leq L_t} (Q^{l_j}-Q^{l_i}) (\tilde{Q}^{l_j} - \tilde{Q}^{l_i}) \prod_{1 \leq i \leq L_t} w_d(l_i)
    \end{equation}
    where the discrete weight $w_d$ is given by 
    \begin{equation}
        w_d(x) = \begin{dcases}
            a^{x} (Q \tilde{Q})^{\frac{x}{2}} Q^{|t|} (\tilde{Q}^{x-|t|+1}; \tilde{Q})_{N-(M-|t|)} & \text{if } t \leq 0, \\
            a^{x} (Q \tilde{Q})^{\frac{x}{2}} \tilde{Q}^{t} (\tilde{Q}^{x+1}; \tilde{Q})_{N-t-M} & \text{if } t > 0 \text{ and } N-t \geq M, \\
            a^{x} (Q \tilde{Q})^{\frac{x}{2}} \tilde{Q}^{t} (Q^{x+N-t-M+1}; Q)_{M-(N-t)} & \text{if } t > 0 \text{ and } N-t < M.
        \end{dcases}
    \end{equation}
\end{thm}

\begin{rem}
    The normalization constant above can be made explicit (for example, by using the Cauchy--Binet formula).
\end{rem}

\begin{rem}
    By computing the weight $w^{d}$ in some special cases we obtain the following limits: 
    \begin{itemize}
        \item (\emph{Meixner}) when $\eta = \theta = 0$ ($Q = \tilde{Q} = 1$), the ensemble $l^{(t)}$ above is Johansson's~\cite{joh00} Meixner orthogonal polynomial ensemble appearing in point-to-point directed last passage percolation with geometric weights;
        \item (\emph{little $q$-Jacobi}) when $a=1$ and $0 < \eta = \theta$ ($=1$ would suffice) we obtain the little $q$-Jacobi orthogonal polynomial ensemble. This latter has not appeared before to the best of our knowledge. It is the orthogonal polynomial ensemble behind $q^{\rm volume}$-weighted plane partitions studied in~\cite{bnh84, ck01, or03} (see the latter for the explicit determinantal connections to our work).
    \end{itemize}
\end{rem}

As was already observed by Muttalib~\cite{mut95}, each ensemble $l^{(t)}$ is a determinantal bi-orthogonal ensemble. In fact, the whole extended process is (unsurprisingly) determinantal with an explicit (discrete time-) extended correlation kernel. This is our next result.

\begin{thm} \label{thm:disc_corr}
    Fix a positive integer $n$, discrete times $t_1 < \dots < t_n$ between $-M+1$ and $N-1$, and non-negative particle positions $k_i, 1 \leq i \leq n$. The process $(l^{(t)})_{-M+1 \leq t \leq N-1}$ is determinantal: 
    \begin{equation}
        \P(l^{(t_i)} \text{ has a particle at position } k_i, \ \forall \ 1 \leq i \leq n) = \det_{1 \leq i, j \leq n} K_d (t_i, k_i; t_j, k_j)
    \end{equation}
    where the extended correlation kernel $K_d$ is given by 
    \begin{equation}
        K_d (s,k; t,\ell) = \oint\limits_{|z|=1+\delta} \!\!\! \frac{dz}{2 \pi i} \!\!\! \oint\limits_{|w|=1-\delta} \!\!\! \frac{dw}{2 \pi i}  \frac{F_d(s,z)}{F_d(t,w)} \frac{w^{\ell-M}}{z^{k-M+1}} \frac{1}{z-w} - \Id_{[s > t]} \!\!\! \oint\limits_{|z|=1+\delta} \!\!\! \frac{dz}{2 \pi i} \frac{F_d(s,z)}{F_d(t,z)} z^{\ell-k-1} 
    \end{equation}
    with
    \begin{equation}
        F_d(s,z) = \begin{dcases}
            \frac{(\sqrt{a} \tilde{Q}^{1/2}/z; \tilde{Q})_{N}}{(\sqrt{a} Q^{|s|+1/2}z; Q)_{M-|s|}}, & \text{if } s \leq 0, \\
            \frac{(\sqrt{a} \tilde{Q}^{s + 1/2}/z; \tilde{Q})_{N-s}}{(\sqrt{a} Q^{1/2}z; Q)_{M}}, & \text{if } s > 0
        \end{dcases} 
        = \frac{(\sqrt{a} \tilde{Q}^{s^+ + 1/2}/z; \tilde{Q})_{N-s^+}}{(\sqrt{a} Q^{s^- + 1/2}z; Q)_{M-s^-}} 
    \end{equation}
    and with $\delta$ small enough so that the $w$ contour contains all finitely many poles of the integrand of the form $\sqrt{a} \tilde{Q}^{1/2+\cdots}$ and the $z$ contour excludes all finitely many poles of the form $\left( \sqrt{a} Q^{1/2+\cdots}\right)^{-1}$---see Figure~\ref{fig:K_d_cont}.
\end{thm}

\begin{rem}
    The kernel $K_d$ has two alternative formulas, one of which is an explicit sum of basic hypergeometric type---see Proposition~\ref{prop:K_d_alt_2}. 
\end{rem}

\begin{rem} 
    We notice that $\Id_{[s > t]} \frac{F_d(s,z)}{F_d(t, z)}$ simplifies considerably:
    \begin{equation} \label{eq:V_int_d}
        \begin{split}
        \Id_{[s > t]} \frac{F_d(s,z)}{F_d(t, z)} &= \begin{dcases}
            \frac{1}{(\sqrt{a} \tilde{Q}^{t+1/2}/z; \tilde{Q})_{s-t}} & \text{if } s > t \geq 0, \\
            \frac{1}{(\sqrt{a} Q^{|s|+1/2} z; Q)_{|t|-|s|}} & \text{if } 0 \geq s > t, \\
            \frac{1}{(\sqrt{a} \tilde{Q}^{1/2}/z; \tilde{Q})_{s} (\sqrt{a} Q^{1/2} z; Q)_{|t|} } & \text{if } s \geq 0 > t
        \end{dcases} \\
        &=  \frac{\Id_{[s > t]}}{(\sqrt{a} \tilde{Q}^{t^+ + 1/2}/z; \tilde{Q})_{s^+ - t^+} (\sqrt{a} Q^{s^- + 1/2} z; Q)_{t^- - s^-} }.
    \end{split}
    \end{equation}
\end{rem}

\subsection{Continuous Muttalib--Borodin processes and the hard edge limit} \label{sec:cont_mbe}

In this section we obtain a Jacobi-like Muttalib--Borodin process (recall it means each slice is an MB ensemble) as a $q \to 1$ limit of the discrete results above. Each slice from this process is a vector $\x = (x_1 < \dots < x_n)$ of real numbers almost surely strictly between 0 and 1 ordered increasingly---note this is the opposite convention to that used for partitions. We call as before $n = \ell(\x)$ the length of $\x$. Two such vectors are interlacing and we write $\y \prec \x$ if 
\begin{equation} \label{eq:cont_interlacing}
    x_n > y_n > x_{n-1} > y_{n-1} > \cdots
\end{equation}  
where we note again we use the opposite interlacing convention for continuous vectors than we did for partitions. Conventionally, it is sometimes useful to consider a 0-th part $x_0 = 0$ and infinitely many parts (padded after $x_n$) $x_{n+1} = x_{n+2} = \dots = 1$ much like partitions $\lambda$ can be padded with infinitely many trailing zeros and we can write $\lambda_0 = \infty$ whenever the situation requires it.

\begin{figure} [!t]
    \begin{center}
        \includegraphics[scale=0.5]{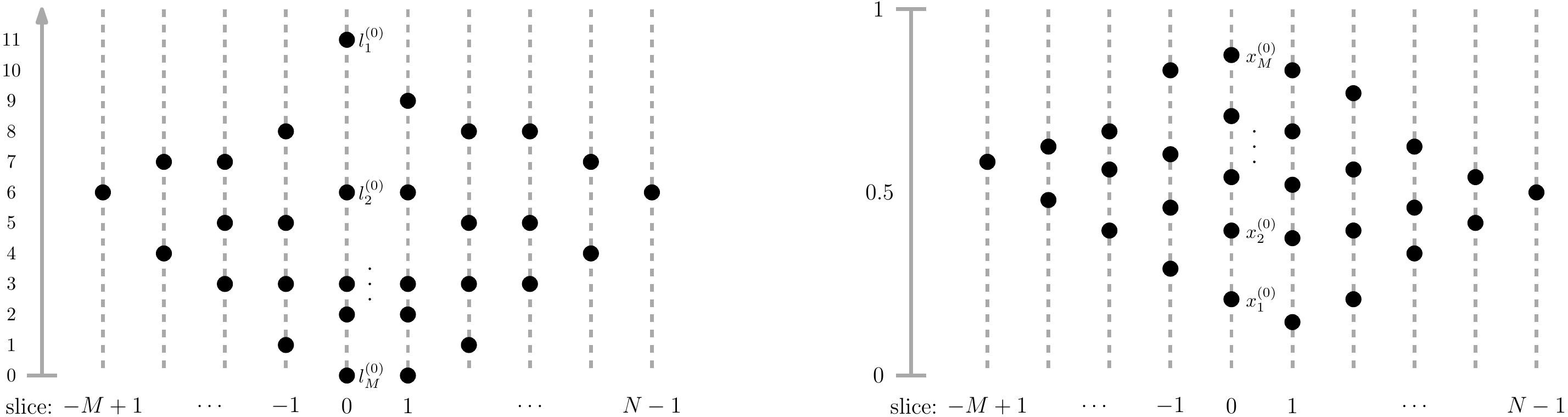}
    \end{center}
    \caption{Left: the discrete process $(l^{(t)})_{-M+1 \leq t \leq N-1}$ corresponding to the lozenges of Figure~\ref{fig:pp_mb}. Right: a possible instantiation of the continuous process $(\x^{(t)})_{-M+1 \leq t \leq N-1}$. Note we index the particles differently in the continuous and discrete settings.}
    \label{fig:proc_mb}
\end{figure}

\subsubsection{Finite-size results}

Our first result, a time-extension of~\cite[Prop.~6]{fr05}, is the construction of $(\x^{(t)})_t$. A possible instantiation of this limit process is given in Figure~\ref{fig:proc_mb} (right).

\begin{thm} \label{thm:cont_mb_dist}
Fix $\alpha \geq 0$. In the following $q \to 1-$ limit:
\begin{equation}
    q = e^{-\epsilon}, \quad a = e^{-\alpha \epsilon}, \quad \lambda^{(t)}_i = -\frac{\log x_i^{(t)}}{\epsilon}, \quad \epsilon \to 0+
\end{equation}
the discrete-space Muttalib--Borodin process $(l^{(t)})_{-M+1 \leq t \leq N-1}$ converges, in the sense of weak convergence of finite dimensional distributions, to a continuous-space (and discrete-time) process ${\bf X} = (\x^{(t)})_{-M+1 \leq t \leq N-1}$ where each time slice $\x^{(t)}$ has $L_t$ points supported in $[0,1]$, the slices are interlacing as follows:
\begin{equation} \label{eq:cont_interlacing_2}
    \emptyset \prec \x^{(-M+1)} \prec \cdots \prec \x^{(-1)} \prec \x^{(0)} \succ \x^{(1)} \succ \cdots \succ \x^{(N-1)} \succ \emptyset
\end{equation}
and each $\x^{(t)}$ is distributed according to the following Muttalib--Borodin measure:
\begin{equation}
    \P(\x^{(t)} = \x) dx_1 \dots dx_{L_t} \propto \prod_{1 \leq i < j  \leq L_t} (x_j^\eta-x_i^\eta) (x_j^\theta-x_i^\theta) \prod_{1 \leq i \leq L_t} w_c(x_i) dx_i
\end{equation}
where
\begin{equation}
    w_c(x) = \begin{dcases}
        x^{\alpha + \frac{\eta + \theta}{2} + |t| \eta - 1} (1-x^\theta)^{N-(M-|t|)} & \text{if } t \leq 0, \\
        x^{\alpha + \frac{\eta + \theta}{2} + t \theta - 1} (1-x^\theta)^{N-t-M} & \text{if } t > 0 \text{ and } N-t \geq M, \\
        x^{\alpha + \frac{\eta + \theta}{2} + t \theta - 1} (1-x^\eta)^{M-(N-t)} & \text{if } t > 0 \text{ and } N-t < M.
    \end{dcases}
\end{equation}
\end{thm}

\begin{rem}
    The weight $w_c$ is a deformation of the classical Jacobi weight. Moreover the case $t=0, M=N, \theta=1$ recovers, after some reparametrization, the Jacobi-like Muttalib--Borodin ensemble considered by Borodin~\cite{bor99}.
\end{rem}

The process thus obtained is determinantal, which is our next result.

\begin{thm} \label{thm:cont_corr}
    The Muttalib--Borodin process $(\x^{(t)})_{-M+1 \leq t \leq N-1}$ has determinantal correlations. Fix $n$ a positive integer, $t_1 < \dots < t_n$ discrete times between $-M+1$ and $N-1$, and $x_k \in (0,1) \ \forall 1 \leq k \leq n$.
\begin{equation}
    \P \left[\bigcap_{k=1}^n \left\{ \text{slice } \x^{(t_k)} \text{ has a particle in } (x_k, x_k+dx_k) \right\} \right] \prod_{k=1}^{n} dx_k = \det_{1 \leq k, \ell \leq n} K_c(t_k, x_k; t_\ell, x_\ell) \prod_{k=1}^{n} dx_k
\end{equation}
with extended correlation kernel given by
\begin{equation}
    K_c (s, x; t, y) = \frac{1}{\sqrt{xy}} \int\limits_{\delta + i \R} \frac{d \zeta}{2 \pi i} \int\limits_{-\delta + i \R} \frac{d \omega}{2 \pi i} \frac{F_c(s, \zeta)}{F_c(t, \omega)} \frac{x^\zeta}{y^\omega} \frac{1} {\zeta-\omega} - \frac{\Id_{[s > t]}}{\sqrt{xy}} \int\limits_{\delta + i \R} \frac{d \zeta}{2 \pi i} \frac{F_c(s, \zeta)}{F_c(t, \zeta)} (x y^{-1})^{\zeta}
\end{equation}
where
\begin{equation}
    \begin{split}
    F_c(s, \zeta) = \begin{dcases}
        \frac{\theta^N}{\eta^{M-|s|}} \cdot \frac{(\frac{\alpha}{2 \theta} + \frac{\zeta}{\theta} + \frac{1}{2})_{N} } {(\frac{\alpha}{2 \eta} - \frac{\zeta}{\eta} + |s| + \frac{1}{2})_{M-|s|}} & \text{if } s \leq 0, \\
        \frac{\theta^{N-s}}{\eta^{M}} \cdot\frac{(\frac{\alpha}{2 \theta} + \frac{\zeta}{\theta} + s + \frac{1}{2})_{N-s} } {(\frac{\alpha}{2 \eta} - \frac{\zeta}{\eta} + \frac{1}{2})_{M}} & \text{if } s > 0
    \end{dcases}
    = \frac{ \theta^{N - s^+}}{\eta^{M - s^-}} \cdot \frac{(\frac{\alpha}{2 \theta} + \frac{\zeta}{\theta} + s^+ + \frac{1}{2})_{N - s^+} } {(\frac{\alpha}{2 \eta} - \frac{\zeta}{\eta} + s^- + \frac{1}{2})_{M - s^-}}
    \end{split}
\end{equation}
with $(x)_n = \prod_{0 \leq i < n} (x+i)=\frac{\Gamma(x+n)}{\Gamma(x)}$ the Pochhammer symbol and where the contours are bottom-to-top oriented vertical lines such that all finitely many poles of the integrands of the form $-\frac{\alpha}{2}-\frac{\theta}{2}-\cdots$ lie to the left and all finitely many poles of the form $\frac{\alpha}{2}+\frac{\eta}{2}+\cdots$ lie on the right (any $0 < \delta < \tfrac12$ would do) of both.
\end{thm}

\begin{rem} \label{rem:cont_K_c}
    Let us make an important remark on the contours in the double contour integral part of $K_c$. They are improperly written above. One choice is to take them as closed enclosing only the relevant poles: these are the closed contours (and essentially the only content) of Proposition~\ref{prop:K_c_alt_1} and of Figure~\ref{fig:K_c_cont}, up to reversing the $\zeta$ orientation. Because of the $\zeta$ decay at real $\infty$ (for $x^\zeta$) and the $\omega$ decay at real $-\infty$ (for $y^{-\omega}$), we can open up the $\zeta$ and $\omega$ contours at $\infty$ and $-\infty$ respectively, starting from the closed ones, without changing the value of the complex integral. They become the Hankel contours of Figure~\ref{fig:K_c_cont}. We can then move these latter ones on the Riemann sphere to obtain the vertical contours described in the statement of the theorem above. A similar remark holds for the single contour integral of $K_c$. See also Proposition~\ref{prop:K_c_alt_1}. The reason to use vertical contours is purely of convenience, as it makes the hard-edge limit of Theorem~\ref{thm:he_kernel} transparent, and it also allows us to write the kernel in alternative ways using known formulas for the Fox $H$-function---see Proposition~\ref{prop:K_c_alt_0}. Finally, the three important features of any contour choices in the double contour integral are: that they do not intersect; that they enclose only the relevant (finitely many) poles; and that $\zeta$ is to the right of $\omega$. 
\end{rem}

\begin{rem}
    We notice that $\Id_{[s > t]} \frac{F_c(s,\zeta)}{F_c(t, \zeta)}$ simplifies considerably:
    \begin{equation} \label{eq:V_int_c}
        \begin{split}
        \Id_{[s > t]} \frac{F_c(s,\zeta)}{F_c(t, \zeta)} &= 
        \begin{dcases}
            \frac{\theta^{t-s}}{(\frac{\alpha}{2 \theta} + \frac{\zeta}{\theta} + t + \frac{1}{2})_{s-t}} & \text{if } s > t \geq 0, \\
            \frac{\eta^{|s|-|t|}}{(\frac{\alpha}{2 \eta} - \frac{\zeta}{\eta} + |s| + \frac{1}{2})_{|t|-|s|}} & \text{if } 0 \geq s > t, \\
            \frac{\eta^{-|t|} \theta^{-s}}{(\frac{\alpha}{2 \theta} + \frac{\zeta}{\theta} + \frac{1}{2})_{s} (\frac{\alpha} {2 \eta} - \frac{\zeta}{\eta} + \frac{1}{2})_{|t|} } & \text{if } s \geq 0 > t 
        \end{dcases} \\
        &= \frac{\Id_{[s > t]} \eta^{s^- - t^-} \theta^{t^+ - s^+}} {(\frac{\alpha}{2 \theta} + \frac{\zeta}{\theta} + t^+ + \frac{1}{2})_{s^+ - t^+} (\frac{\alpha} {2 \eta} - \frac{\zeta}{\eta} + s^- + \frac{1}{2})_{t^- - s^-} }.
    \end{split}
    \end{equation}
\end{rem}

\begin{rem}
    It is not obvious from the formula we have for $K_c$ that it reduces to Borodin's finite kernel from~\cite[Prop.~3.3]{bor99} in the case $s=t=0, M=N, \eta = 1$. This is indeed the case and is explained in Remark~\ref{rem:bor_lim_fin} right after we rewrite $K_c$ in a more suitable form in Proposition~\ref{prop:K_c_alt_2} of Appendix~\ref{sec:alt_c}. 
\end{rem}

\begin{figure}[!h]
    \begin{center}
        \includegraphics[scale=0.5]{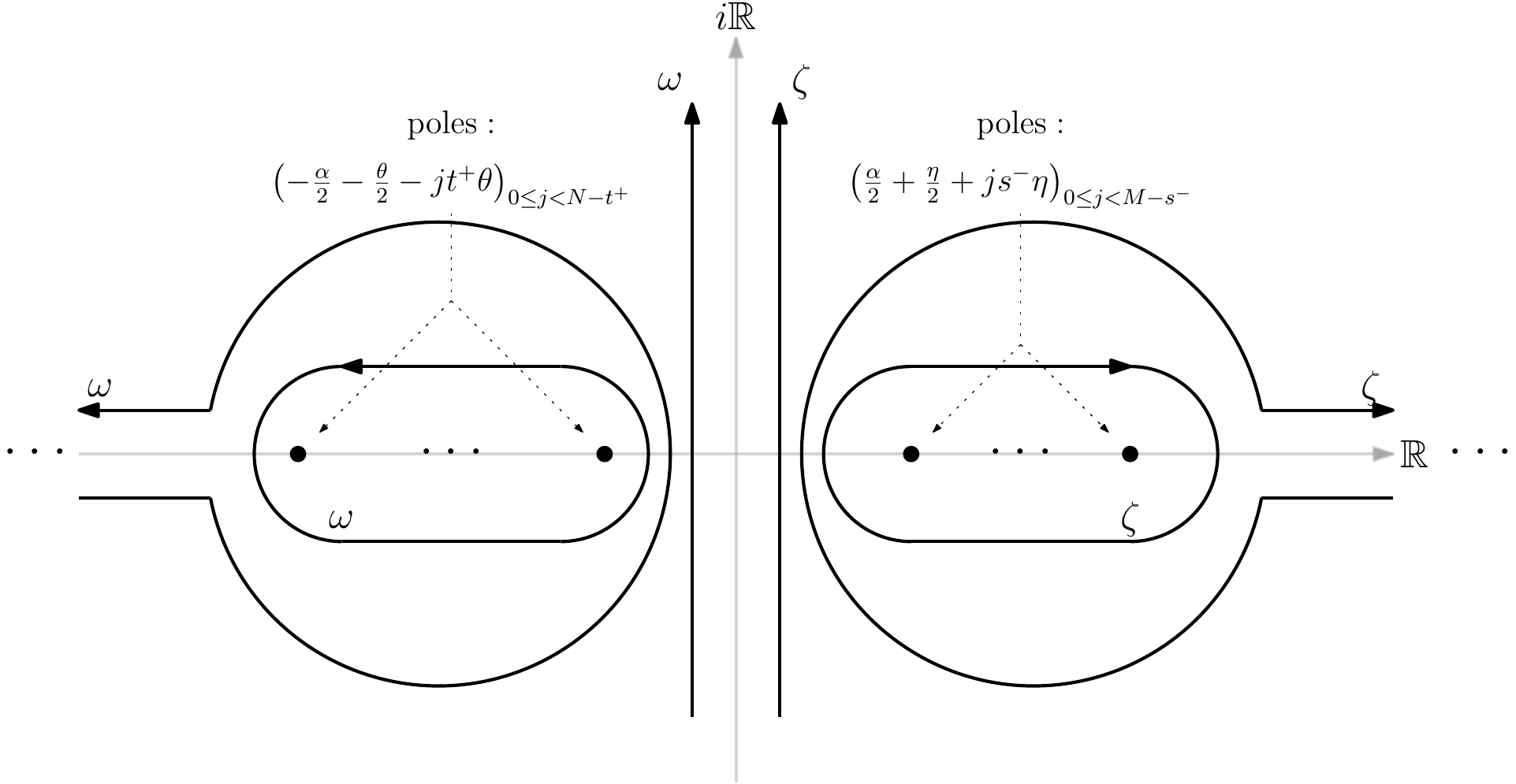}
    \end{center}
    \caption{Three choices of contours for the double contour integral part of Theorem~\ref{thm:cont_corr}. The ones in the statement are the vertical ones; properly one starts with the closed ones encircling only the relevant poles and opens them up at $\infty$ (for $\zeta$) and $-\infty$ (for $\omega$) to obtain the other two sets. The intermediate Hankel contours also appear sometimes in the literature, so we decided to include depict them as well.}
    \label{fig:K_c_cont}
\end{figure}

\subsubsection{The hard-edge limit}

Finally, let us consider the so-called ``hard-edge scaling'' of the above process. We will do so around the hard-edge of the support $x \approx 0$. We do not have good arguments to handle scaling around $x \approx 1$. 

The first result concerns the hard-edge limit of the kernel. We will let $M, N \to \infty$, and it turns out we can take the limits independently and look what happens to the point process at discrete finite times (otherwise said, we focus our attention around time 0). 


\begin{thm} \label{thm:he_kernel}
    Consider the kernel $K_c(s, x; t, y)$ for fixed $s, t$. We have the following limit:
    \begin{equation}
        \lim_{M, N \to \infty} \frac{1}{M^{\frac{1}{\eta}} N^{\frac{1}{\theta}}} K_c \left(s, \frac{x}{M^{\frac{1}{\eta}} N^{\frac{1}{\theta}}}; t, \frac{y}{M^{\frac{1}{\eta}} N^{\frac{1}{\theta}}} \right) = K_{he} (s, x; t, y)
    \end{equation}
    where the kernel $K_{he}$ (``he'' for hard-edge) is given by
    \begin{equation}
    K_{he} (s, x; t, y) = \frac{1}{\sqrt{xy}} \int\limits_{\delta + i \R} \frac{d \zeta}{2 \pi i} \int\limits_{-\delta + i \R} \frac{d \omega}{2 \pi i} \frac{F_{he}(s, \zeta)}{F_{he}(t, \omega)} \frac{x^\zeta}{y^\omega} \frac{1} {\zeta-\omega} - \frac{\Id_{[s > t]}}{\sqrt{xy}} \int\limits_{\delta + i \R} \frac{d \zeta}{2 \pi i} \frac{F_{he}(s, \zeta)}{F_{he}(t, \zeta)} (x y^{-1})^{\zeta}
    \end{equation}
    with 
    \begin{equation}
        F_{he} (s, \zeta) = 
        \begin{dcases}
            \eta^{|s|} \frac{\Gamma(\frac{\alpha} {2 \eta} - \frac{\zeta}{\eta} + |s| + \frac{1}{2})} {\Gamma(\frac{\alpha}{2 \theta} + \frac{\zeta}{\theta} + \frac{1}{2})} & \text{if } s \leq 0, \\
            \theta^{-s} \frac{\Gamma(\frac{\alpha} {2 \eta} - \frac{\zeta}{\eta} + \frac{1}{2})} {\Gamma(\frac{\alpha}{2 \theta} + \frac{\zeta}{\theta} + s + \frac{1}{2})} & \text{if } s > 0
        \end{dcases}
        = \frac{\eta^{s^-}} { \theta^{s^+} } \frac{\Gamma(\frac{\alpha} {2 \eta} - \frac{\zeta}{\eta} + s^- + \frac{1}{2})} {\Gamma(\frac{\alpha}{2 \theta} + \frac{\zeta}{\theta} + s^+ + \frac{1}{2})}
    \end{equation}
    where $\delta$ is small enough so that all (now possibly infinitely many) poles of the integrands of the form $-\frac{\alpha}{2}-\frac{\theta}{2}-\cdots$ lie to the left of both contours and all poles of the form $\frac{\alpha}{2}+\frac{\eta}{2}+\cdots$ lie to the right.
\end{thm}

\begin{rem} Observe that 
\begin{equation}
    \Id_{[s > t]} \frac{F_{he}(s, \zeta)} {F_{he}(t, \zeta)} = \Id_{[s > t]} \frac{F_c(s, \zeta)}{F_c(t, \zeta)}
\end{equation}
and the right-hand side is given explicitly in equation~\eqref{eq:V_int_c}.
\end{rem}

\begin{rem} \label{rem:cont_K_he}
    A similar statement to that of Remark~\ref{rem:cont_K_c} applies to these contours as well. For the double contour integral in $K_{he}$, they are depicted in Figure~\ref{fig:K_he_cont}. We use the vertical ones though sometimes it is convenient (for numerical evaluation perhaps) to use the Hankel contours also depicted in fig.~cit. It is important that they do not intersect, they ``enclose'' all the poles, and that $\zeta$ is to the right of $\omega$.
\end{rem}

\begin{figure}[!h]
    \begin{center}
        \includegraphics[scale=0.5]{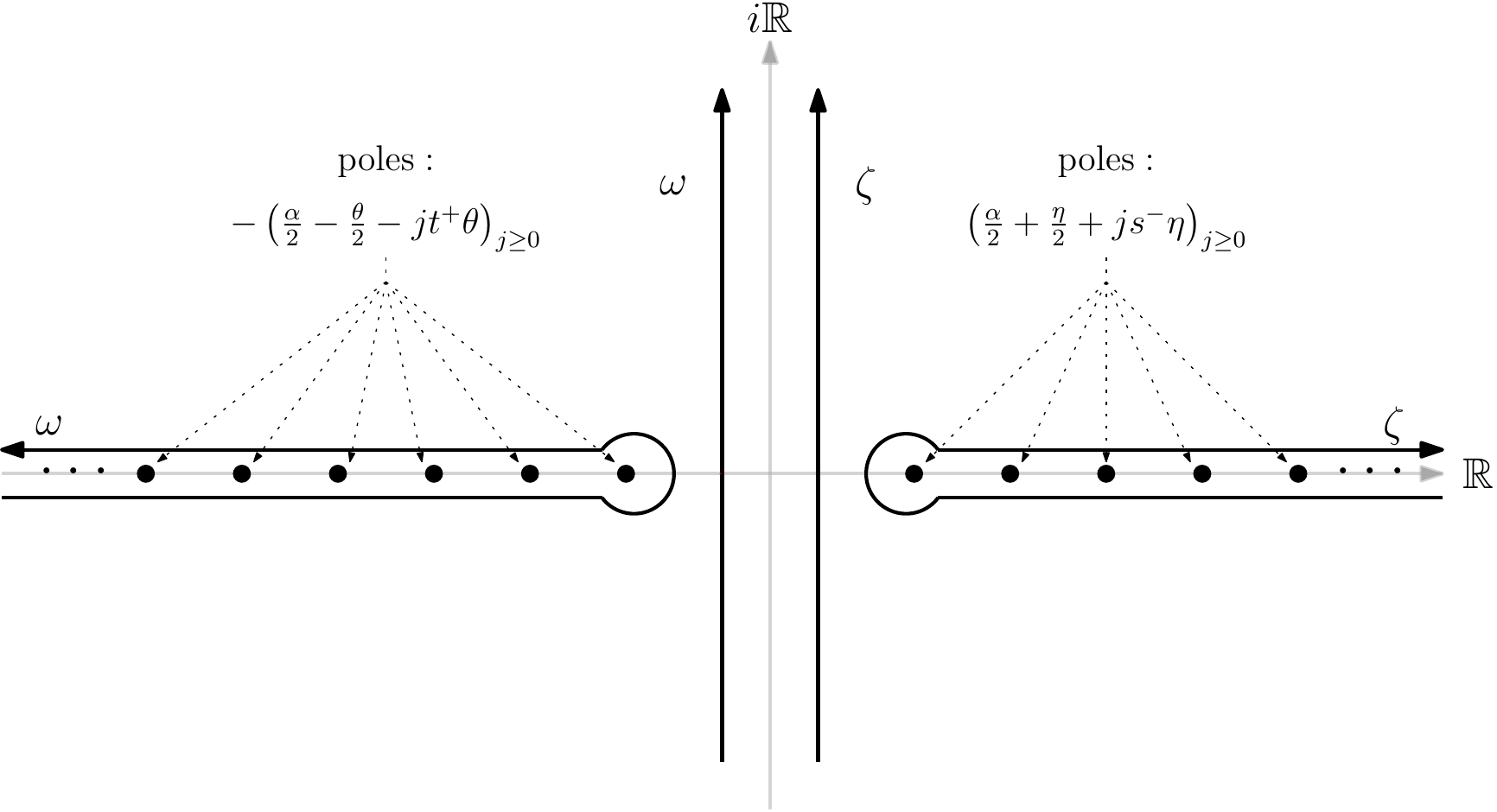}
    \end{center}
    \caption{Two choices of contours for the double contour integral part of $K_{he}$ in Theorem~\ref{thm:cont_corr}. The ones in the statement are the vertical ones, and these often appear in the definition of the Fox $H$-function below; the Hankel contours also sometimes appear in the literature so we included them. There are infinitely many poles in both $\zeta$ and $\omega$.}
    \label{fig:K_he_cont}
\end{figure}

\begin{rem}
    When $s = t = 0$ and  $\eta = 1$, up to conjugation, $K_{he}(0, x; 0, y)$ agrees with Borodin's hard-edge kernel $K_B$ from~\eqref{eq:bor_ker}. The proof of this is deferred to Appendix~\ref{sec:alt_he}, notably Proposition~\ref{prop:K_he_alt_2} and Remark~\ref{rem:bor_lim_he}.
\end{rem}

\begin{prop}
    For $\eta = \theta = 1$, the kernel $K_{he}(0, x; 0, y)$ becomes the hard-edge Bessel kernel of random matrix theory~\cite{tw94_bessel, for93}. We have
    \begin{equation} \label{eq:bessel_kernel}
        K_{he} (0, x; 0, y) = K_{\alpha, \rm Bessel} (x, y) = \int_{0}^1 J_{\alpha} (2 \sqrt{ux}) J_{\alpha} (2 \sqrt{uy}) du
    \end{equation}
    where $J_\alpha$ is the Bessel function of the first kind. 
\end{prop}

\begin{proof}
This can be readily seen as follows: starting from the left, we first use the very simple formula $\frac{x^{\zeta-1/2}y^{-\omega-1/2} } {\zeta - \omega} = \int_0^1 (ux)^{\zeta-1/2} (uy)^{-\omega-1/2} du$ (valid as $\Re(\zeta) > \Re(\omega)$ due to our choice of contours); then each of the $\zeta$ and $\omega$ integrals are the required Bessel functions upon using~\cite[eq.~10.9.22]{nist}
    \begin{equation}
        J_{\alpha} (X) = \frac{1}{2\pi i} \int_{-i\infty}^{i\infty} \frac{\Gamma (-Z)(\tfrac{1}{2}X)^{\alpha+2 Z}} {\Gamma(\alpha+Z+1)}dZ
    \end{equation}
    where the contour passes to the right of $0, 1, 2 \dots$ and where, in the notation above, we use $(Z, X) = \left( \zeta - \frac{\alpha}{2} - \frac{1}{2}, 2 \sqrt{ux} \right)$ for the $(\zeta, x)$ integral and $(Z, X) = \left( -\omega - \frac{\alpha}{2} - \frac{1}{2}, 2 \sqrt{uy} \right)$ for the $(\omega, y)$ integral (in which case the contour also needs to pass to the left of $\N$ around $0$).
\end{proof}

We can furthermore express the whole extended hard-edge limit kernel $K_{he}$ in terms of the Fox $H$-function~\cite[eq.~(51)]{fox61} which we now define. Pick integers $p, q \geq 0$ (not both zero) and integers $0 \leq m \leq q$ and $0 \leq n \leq p$. Pick also complex numbers $a_i, e_i, b_j, c_j$ with $1 \leq i \leq p$, $1 \leq j \leq q$ such that the $c$'s and the $e$'s are positive real numbers. The $H$-function depending on all the above parameters is defined as the following Mellin transform: 
\begin{equation} \label{eq:fox_h}
    H^{m, n}_{q, p} \left[ x \left| \begin{array}{c} (a_1, e_1), \dots , (a_p, e_p) \\ (b_1, c_1), \dots , (b_q, c_q) \end{array} \right] \right.  = \int_T \frac{dz}{2 \pi i} \frac{ \prod_{i=1}^m \Gamma(b_i  + c_i z) \prod_{i=1}^n \Gamma(a_i - e_i z) }   { \prod_{i=m+1}^q \Gamma(b_i  + c_i z) \prod_{i=n+1}^p \Gamma(a_i - e_i z) } x^{-z}
\end{equation}
where we assume that the poles of the numerator of the integrand are all simple and that the contour is a vertical line parallel to the imaginary axis which has the poles of $\Gamma(a_i-e_i z)$ to the right (for all $1 \leq i \leq m$ in the numerator) and those of $\Gamma (b_j  + c_j z)$ to the left (for all $1 \leq j \leq m$ in the numerator). In particular, for $p=q=1$, $(m, n) \in \{(0, 1), (1, 0)\}$ we have:
\begin{equation}
    \begin{split}
    H^{0, 1}_{1, 1} \left[ x \left| \begin{array}{c} (a_1, e_1) \\ (b_1, c_1) \end{array} \right] \right.  & = \int_T \frac{dz}{2 \pi i} \frac{ \Gamma(a_1 - e_1 z) }   { \Gamma(b_1  + c_1 z) } x^{-z}, \\
    H^{1, 0}_{1, 1} \left[ y \left| \begin{array}{c} (a_1, e_1) \\ (b_1, c_1) \end{array} \right] \right.  & = \int_T \frac{d w}{2 \pi i} \frac{ \Gamma(b_1  + c_1 w) }{ \Gamma(a_1 - e_1 w) }  y^{-w}
    \end{split}
\end{equation}
and if $a_1, b_1 > 0$ we can just take $T = i \R$ (or a contour close enough to $i \R$ like the ones from Theorem~\ref{thm:he_kernel}).

Then and as before, by using  the simple formula $\frac{x^{\zeta-1/2}y^{-\omega-1/2} } {\zeta - \omega} = \int_0^1 (ux)^{\zeta-1/2} (uy)^{-\omega-1/2} du$ into the definition of $K_{he}$ and matching the remaining integrals with the appropriate Fox $H$-functions, we arrive at the following rewriting which motivates naming $K_{he}$ (an instance of) the \emph{Fox $H$-kernel}.

\begin{prop} \label{prop:fox_h}
    The hard-edge kernel $K_{he}$ has the following form:
\begin{equation}
    K_{he} (s, x; t, y) = \frac{1}{\sqrt{xy}} \int_0^1 f_{he}^{(s)}\left(\frac{1}{ux}\right) g_{he}^{(t)}(uy) \frac{du}{u} - \Id_{[s > t]} \frac{h(s, x; t, y)}{\sqrt{xy}} 
\end{equation}
where
\begin{equation}
    \begin{split}
    (f_{he}^{(s)}(x), g_{he}^{(s)}(x)) &= \begin{dcases}
        \left( \eta^{|s|} H^{0, 1}_{1, 1} \left[ x \left| \begin{array}{c} \left(\frac{\alpha}{2 \eta} + |s| + \frac{1}{2} , \frac{1}{\eta}\right) \\ \left(\frac{\alpha}{2 \theta} + \frac{1}{2} , \frac{1}{\theta} \right)  \end{array} \right] \right. , \eta^{-|s|} H^{1, 0}_{1, 1} \left[ x \left| \begin{array}{c} \left(\frac{\alpha}{2 \eta} + |s| + \frac{1}{2} , \frac{1}{\eta}\right) \\ \left(\frac{\alpha}{2 \theta} + \frac{1}{2} , \frac{1}{\theta} \right)  \end{array} \right] \right. \right) & \text{if } s \leq 0, \\
        \left( \theta^{-s} H^{0, 1}_{1, 1} \left[ x \left| \begin{array}{c} \left(\frac{\alpha}{2 \eta} + \frac{1}{2} , \frac{1}{\eta}\right) \\ \left(\frac{\alpha}{2 \theta} + s + \frac{1}{2} , \frac{1}{\theta} \right)  \end{array} \right] \right. , \theta^{s}  H^{1, 0}_{1, 1} \left[ x \left| \begin{array}{c} \left(\frac{\alpha}{2 \eta} + \frac{1}{2} , \frac{1}{\eta}\right) \\ \left(\frac{\alpha}{2 \theta} + s + \frac{1}{2} , \frac{1}{\theta} \right)  \end{array} \right] \right. \right) & \text{if } s > 0
    \end{dcases}
    \end{split}
\end{equation}
and where
\begin{equation}
    h(s, x; t, y) =
    \begin{dcases}
        \theta^{t-s} H^{1, 0}_{2, 0} \left[ \frac{y}{x} \left| \begin{array}{c} - \\ \left(\frac{\alpha}{2 \theta} + t + \frac{1}{2} , \frac{1}{\theta}\right), \left(\frac{\alpha}{2 \theta} + s + \frac{1}{2} , \frac{1}{\theta}\right)  \end{array} \right] \right. & \text{if } s > t \geq 0, \\
        \eta^{|s|-|t|} H^{0, 1}_{0, 2} \left[ \frac{y}{x} \left| \begin{array}{c} \left(\frac{\alpha}{2 \eta} + |s| + \frac{1}{2} , \frac{1}{\eta}\right), \left(\frac{\alpha}{2 \eta} + |t| + \frac{1}{2} , \frac{1}{\eta}\right) \\ -  \end{array} \right] \right. & \text{if } 0 \geq s > t, \\
        \eta^{-|t|} \theta^{-s}  H^{1, 1}_{2, 2} \left[ \frac{y}{x} \left| \begin{array}{c} \left(\frac{\alpha}{2 \eta} + \frac{1}{2} , \frac{1}{\eta}\right), \left(\frac{\alpha}{2 \eta} + |t| + \frac{1}{2} , \frac{1}{\eta}\right) \\ \left(\frac{\alpha}{2 \theta} + \frac{1}{2} , \frac{1}{\theta}\right), \left(\frac{\alpha}{2 \theta} + s + \frac{1}{2} , \frac{1}{\theta}\right)  \end{array} \right] \right. & \text{if } s \geq 0 > t.   
    \end{dcases}
\end{equation}
\end{prop}

Finally, we have the usual hard-edge limit/gap probability statement.
\begin{thm} \label{thm:gap_prob_cont}
    Let $M=N$ (for simplicity of stating the result), pick an integer $-N+1 \leq t \leq N-1$ and consider the ensemble $\x^{(t)}$. We have:
    \begin{equation}
        \lim_{N \to \infty} \P \left( \frac{  x_{1}^{(t)} } {N^{\frac{1}{\eta}+\frac{1}{\theta}}} < r \right) = \det (1 - K_{he}(t, \cdot; t, \cdot))_{L^2(0,r)}.
    \end{equation} 
\end{thm}

\begin{rem}
    Theorem~\ref{thm:gap_prob_cont} has an obvious multi-time (and/or multi-interval) extension which we leave as an exercise to the reader.
\end{rem}

It would be interesting to derive Painlev\'e-type differential (and possibly difference in the discrete time variable) equations for the gap probabilities above in a manner similar to that of Tracy--Widom~\cite{tw94_bessel}.

\subsection{Last passage percolation} \label{sec:lpp}

In this section we present two results in the theory of directed last passage percolation (LPP), both resembling a result of Johansson~\cite{joh08}. We say that a random variable $X$ is geometric of parameter $0 \leq q < 1$ 
\begin{equation}
    X \sim {\rm Geom} (q) \quad \text{if}\quad \P(X=k) = (1-q) q^k, \ k=0,1,\dots
\end{equation}
We likewise say $Y$ is a power random variable of parameter $\beta > 0$ 
\begin{equation}
    Y \sim {\rm Pow} (\beta) \quad \text{if}\quad \P(Y \in (x, x+dx)) = \beta x^{\beta-1}, \ x \in [0,1].
\end{equation}
Let us notice $Y \sim {\rm Pow} (\beta)$ can be obtained from $X \sim {\rm Geom} (q)$ in the limit $\epsilon \to 0+$ with $q = e^{-\beta \epsilon}$ and $Y = -\epsilon^{-1} \log X$.

\subsubsection{A discrete result} \label{sec:disc_lpp}

We start with the discrete setting. Let us recall our continuous parameters from the previous section: $q, a, \eta, \theta$, and also recall $Q = q^\eta, \tilde{Q} = q^\theta$. Consider the integer quadrant lattice consisting of points $(i, j)_{i, j \geq 1}$ with coordinates as in Figure~\ref{fig:disc_lpp}. At each point $(i, j)$ on the anti-diagonal $i+j = k + 1$ place a non-negative integer which is a random variable ${\rm Geom}(a Q^{i-1/2} \tilde{Q}^{k-i+1/2})$, independent of the rest. Let
\begin{itemize}
    \item $L^{\rm geo}_1 = $ the longest down-left path from $(1,1)$ to $(\infty, \infty)$, where the \emph{length} of a path is the sum of the integers on it;
    \item $L^{\rm geo}_2 = $ the longest down-right path from $(\infty,1)$ to $(1, \infty)$.
\end{itemize}

\begin{figure}[!h]
    \centering
    \includegraphics[scale=0.5]{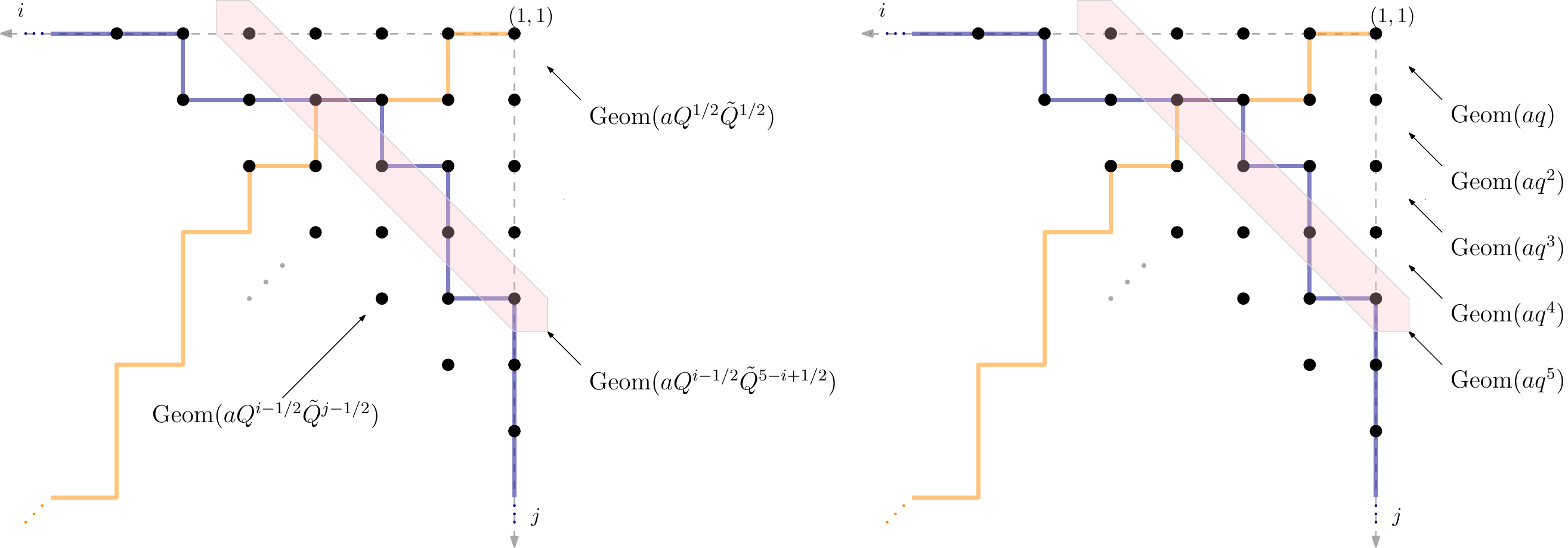} \quad
    \caption{Left: the setting for Theorem~\ref{thm:disc_lpp} and two possible paths that enter into the maxima for $L^{\rm geo}_1$ (orange) and $L^{\rm geo}_2$ (blue). Right: the equi-distributed-by-diagonal case ($Q=\tilde{Q}=q$) where each anti-diagonal $i+j = k+1$ has iid ${\rm Geom}(a q^k)$ random variables on it.}
    \label{fig:disc_lpp}
\end{figure}

By Borel--Cantelli, only finitely many of these geometric random variables are non-zero\footnote{This is equivalent to the finiteness of the partition function for the whole ensemble, and the latter equals 
\begin{equation}
    \prod_{i, j \geq 1} (1-a Q^{i-1/2} \tilde{Q}^{j-1/2})^{-1}.
\end{equation}
}, and thus both $L^{\rm geo}_i, 1 \leq i \leq 2$ are almost surely finite. The setting is depicted in Figure~\ref{fig:cont_lpp}, along with two representative paths/polymers: an orange one for $L^{\rm geo}_1$ and a blue one for $L^{\rm geo}_2$.

We have the following theorem.

\begin{thm} \label{thm:disc_lpp}
    Fix real parameters $\alpha, \eta, \theta \geq 0$ (not all 0). $L^{\rm geo}_1$ and $L^{\rm geo}_2$ are equal in distribution. Moreover, let $L \in \{ L^{\rm geo}_1, L^{\rm geo}_2 \}$. Let 
    \begin{equation}
        q = e^{-\epsilon}, \quad a = e^{-\alpha \epsilon}.
    \end{equation}
    We have:
    \begin{equation}
        \lim_{\epsilon \to 0+} \P \left( \epsilon L + \frac{\log (\epsilon \eta)}{\eta} + \frac{\log (\epsilon \theta)}{\theta} < s \right) = \det(1-\tilde{K}_{he})_{L^2(s, \infty)}
    \end{equation}
    where $\tilde{K}_{he}$ is related to the hard-edge $K_{he}$ of Theorem~\ref{thm:he_kernel} via:
    \begin{equation}
        \tilde{K}_{he} (x, y) = e^{-\frac{x}{2}-\frac{y}{2}} K_{he} (0, e^{-x}; 0, e^{-y}).
    \end{equation}
\end{thm}

\begin{rem}
    Writing $\epsilon = 1/R$, we observe that $L$ has order $O(R \log R)$ and $O(R)$ fluctuations. Contrast this with the $R^{1/3}$ limits of Johansson~\cite{joh00} and compare with similar (exponential) results of~\cite[Thm.~1.1a]{joh08}.
\end{rem}

\begin{rem} \label{rem:interpolation}
    Let us consider the \emph{equi-distributed-by-diagonal} case $\eta = \theta = 1$ (meaning on anti-diagonal $i+j=k+1$ we place $k$ iid ${\rm Geom} (a q^k)$ random variables---see Figure~\ref{fig:disc_lpp} (right)). In that case we have
    \begin{equation}
        \begin{split}
        \tilde{K}_{he} (x, y) &= e^{-\frac{x}{2}-\frac{y}{2}} K_{\alpha, \rm Bessel} (e^{-x}, e^{-y}) \\
                              &= \int\limits_{-\delta+i\R} \frac{d \omega}{2 \pi i} \int\limits_{\delta+i\R} \frac{d \zeta}{2 \pi i}  \frac{\Gamma(\tfrac{\alpha}{2} + \tfrac12 - \zeta)}{\Gamma(\tfrac{\alpha}{2} + \tfrac12 + \zeta)} \frac{\Gamma(\tfrac{\alpha}{2} + \tfrac12 + \omega)} {\Gamma(\tfrac{\alpha}{2} + \tfrac12 - \omega)} \frac{e^{-x\zeta}}{e^{-y\omega}}\frac{1}{\zeta-\omega} 
        \end{split}
    \end{equation}
    ($0 < \delta < \tfrac12$) with $K_{\alpha, \rm Bessel}$ the hard-edge Bessel kernel of~\eqref{eq:bessel_kernel}. Let us write 
    \begin{equation}
        F_{\alpha}(s) = \det(1-\tilde{K}_{he})_{L^2(s, \infty)}
    \end{equation}
    and note the following Gumbel to Tracy--Widom interpolation property of Johansson~\cite{joh08}:
    \begin{itemize}
        \item $\lim_{\alpha \to 0} F_{\alpha}(s) = F_0(s) = e^{-e^{-s}}$ with the latter the Gumbel distribution\footnote{See second equation after (1.8) in~\cite{joh08} but note we believe there is a typo and the equation should read, in Johansson's notation, $U_{-1/2}(s)=\exp(-\exp(-s))$ (the parameter matching between our notation and Johansson's is $\alpha = 2 \beta + 1$). We also have numerical evidence for this using the method of Bornemann~\cite{bor10} for computing Fredholm determinants (and comparing to the Gumbel distribution). This Gumbel result is further consistent with and a mildly weaker form of a result of Vershik--Yakubovich~\cite[Thm.~1]{vy06} (with $c = \beta = 1, x = q$ for the correspondence between their notation and ours); the link here is combinatorial as $L$ is in distribution the same as the largest part of a $q^{\rm volume}$ distributed plane partition with unrestricted base (in other words, $M=N=\infty$ in the notation of Section~\ref{sec:pp}). What \emph{we do not have} is a \emph{direct proof} intrinsic to our results that $F_0(s)=\exp(-\exp(-s))$ is the Gumbel distribution.};
        \item $\lim_{\alpha \to \infty} F_{\alpha}(-2 \log(2(\alpha-1)) + (\alpha-1)^{-2/3} s) = F_{\rm TW}(s)$ with the latter the Tracy--Widom GUE distribution~\cite{tw94_airy}\footnote{See~\cite{joh08}, first equation after (1.8), with $\alpha=2 \beta+1$ matching our notation to his. This result follows from the work of Borodin--Forester~\cite{bf03} and can be seen directly via Nicholson's approximation~\cite[Thm.~2.27]{rom15} that $M^{1/3} J_{2M+x M^{1/3}} (2M) \to Ai(x), \ M \to \infty$ with $Ai$ the Airy function, showing the Bessel kernel converges to the Airy kernel.}.
    \end{itemize}
\end{rem}

\begin{rem}
    For another distribution---the \emph{finite-temperature Tracy--Widom distribution}---interpolating between Gumbel and Tracy--Widom GUE, see Johansson's paper~\cite{joh07} for a random matrix model with largest eigenvalue converging to said distribution, and~\cite{bb19} for a discrete LPP-like model (and references therein for further models related to the Kardar--Parisi--Zhang equation). We further note that the Gumbel distribution appears universally in the study of maxima of iid random variables, while Tracy--Widom GUE in the study of maxima of correlated (often determinantal) random variables like largest eigenvalues of hermitian random matrices. We do not have a very good understanding for this fact in our model, certainly not on the Gumbel side, but we do have direct and heuristic evidence.
\end{rem}


\subsubsection{A continuous result} \label{sec:cont_lpp}

For the continuous result, we use the following parameters: $\alpha \geq 0, \eta, \theta > 0$ together with an integer parameter $N \geq 1$. On the lattice $(i, j)_{N \geq i, j \geq 1}$ place, at $(i, j)$ on the diagonal $i+j=k+1$, a random variable $\omega^{\rm pow}_{i,j} \sim {\rm Pow} (\alpha + \eta(i - \tfrac12) + \theta (k-i+\tfrac12))$, independent of the rest. Let
\begin{itemize}
    \item $L^{\rm pow}_1 = \min\limits_{\pi} \prod\limits_{(i,j) \in \pi} \omega^{\rm pow}_{i,j}$ where the minimum is over all down-left paths $\pi$ from $(1,1)$ to $(N, N)$;
    \item $L^{\rm pow}_2 = \min\limits_{\varpi} \prod\limits_{(i,j) \in \varpi} \omega^{\rm pow}_{i,j}$ where the minimum is over all down-right paths $\varpi$ from $(N, 1)$ to $(1, N)$.
\end{itemize}
The setting is depicted in Figure~\ref{fig:cont_lpp}, along with two representative paths/polymers: $\pi$ for $L^{\rm pow}_1$ (in orange) and $\varpi$ for $L^{\rm pow}_2$ (in blue). Notice that each path picks exactly $2N-1$ strictly positive (and $< 1$) random variables in the product under minimization. By contrast, in the geometric setting we consider sums of infinitely many random (integer) variables---the geometry is infinite, but only finitely many of these latter numbers are non-zero as explained in the previous section. 

\begin{figure}[!h]
    \centering
    \includegraphics[scale=0.5]{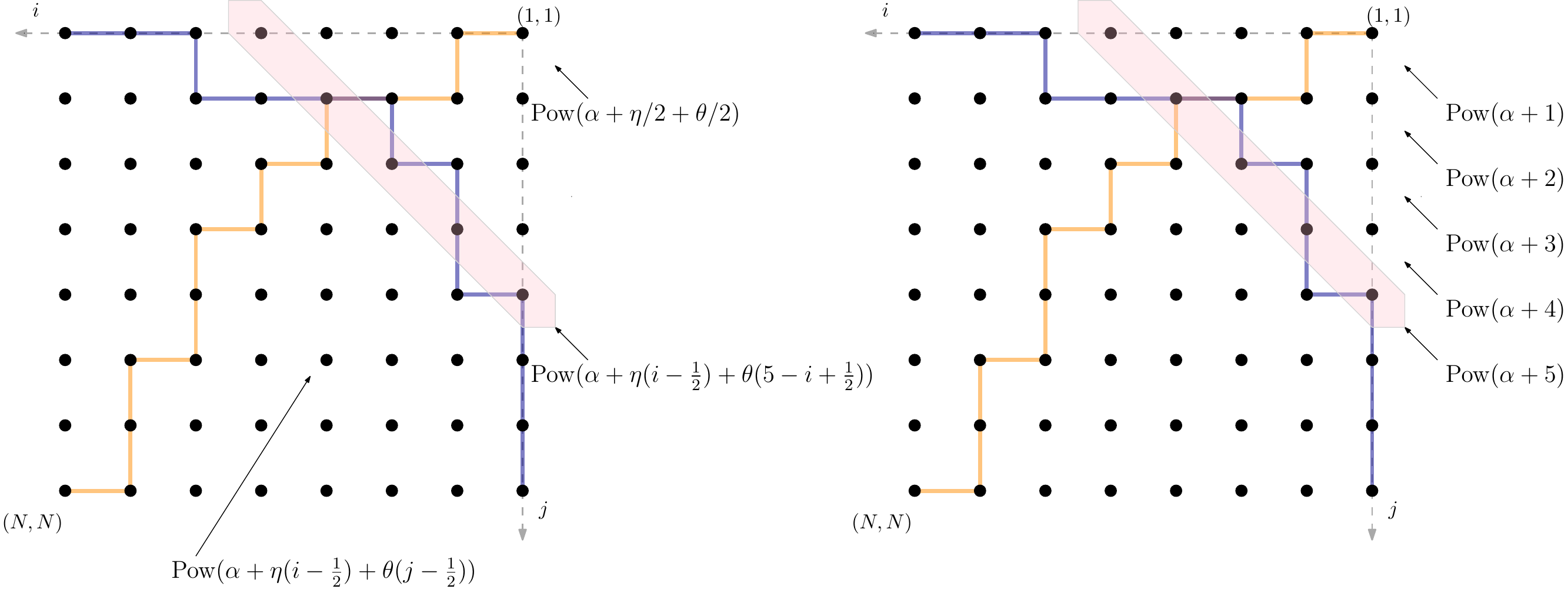} \quad
    \caption{Left: the setting for Theorem~\ref{thm:cont_lpp} and two possible paths that enter into the minima for $L^{\rm pow}_1$ (orange) and $L^{\rm pow}_2$ (blue). Right: the equi-distributed-by-diagonal case ($\eta = \theta = 1$) where each anti-diagonal $i+j = k+1$ has iid ${\rm Pow}(\alpha + k)$ random variables on it.}
    \label{fig:cont_lpp}
\end{figure}

We have the following result, an extension in exponential coordinates of~\cite[Thm.~1.1a]{joh08}.

\begin{thm} \label{thm:cont_lpp}
    $L^{\rm pow}_1$ and $L^{\rm pow}_2$ are equal in distribution. Moreover, if $L \in \{ L^{\rm pow}_1, L^{\rm pow}_2 \}$, we have
    \begin{equation}
        \lim_{N \to \infty} \P \left( \frac{L}{N^{\frac{1}{\eta}+\frac{1}{\theta}}} < r \right) = \det (1-K_{he}(0, \cdot; 0, \cdot))_{L^2 (0,r)}
    \end{equation} 
    where $K_{he} (0, x; 0, y)$ is the time 0 hard-edge kernel of Theorem~\ref{thm:he_kernel}.
\end{thm}

\section{Proofs of Theorems~\ref{thm:disc_mb_dist} and~\ref{thm:disc_corr}} \label{sec:disc_proofs}

In this section we use the technology of \emph{Schur measures and processes}~\cite{oko01, or03} of Okounkov and Reshetikhin to prove the announced results. We do not go into the details as by now there are many places where these tools have been used, (re)proven, and extended: see~\cite{oko01, or03, br05, bbnv18, bb19} for a sample of results and generalizations.

We start with proving Theorem~\ref{thm:disc_mb_dist}. 
\begin{proof}[Proof of Thm.~\ref{thm:disc_mb_dist}]
    The measure~\eqref{eq:pp_measure} comes from a~\emph{Schur process}~\cite{or03}. More precisely, representing a plane partition $\Lambda$ as a sequence of interlacing partitions $(\lambda^{(t)})_{-M \leq t \leq N}$ (extremities being $\emptyset$) as in~\eqref{eq:disc_interlacing}, we have
    \begin{equation} \label{eq:schur_proc}
        \P (\Lambda) = Z^{-1} \prod_{i=0}^{M-1} s_{\lambda^{(-i)} / \lambda^{(-i-1)}} (\sqrt{a} Q^{i+1/2}) \prod_{i=0}^{N-1} s_{\lambda^{(i)} / \lambda^{(i+1)}} (\sqrt{a} \tilde{Q}^{i+1/2})
    \end{equation}
    where $Z = \prod_{i=1}^M \prod_{j=1}^N (1-a Q^{i-1/2} \tilde{Q}^{j-1/2})^{-1}$ is the partition function, $(Q, \tilde{Q}) = (q^\eta, q^\theta)$, and $s_{\lambda / \mu}$ are skew Schur polynomials (functions)~\cite[Ch.~I.5]{mac} which, when evaluated in one variable, give the desired contributions to the measure as $s_{\lambda / \mu} (x) = x^{|\lambda| - |\mu|} \Id_{\mu \prec \lambda}$ (recall that ${\rm left\ vol} = \sum_{i=-M}^{-1} |\lambda^{(i)}|, {\rm central\ vol} = |\lambda^{(0)}|, {\rm right\ vol} = \sum_{i=1}^{N} |\lambda^{(i)}|$). 

    As such, any marginal distribution of a $\lambda^{(t)}$ is a Schur measure~\cite{oko01}. Let us explain this and prove the result for $-M+1 \leq s = -t \leq 0$ (the proof for positive times $s > 0$ follows along the same lines). Precisely we have~\cite{or03} (below $t>0$ and we look at time $-M+1 \leq -t \leq 0$):
    \begin{equation}
        \begin{split}
        \P (\lambda^{(-t)} = \lambda) &= Z^{-1} s_{\lambda} (\sqrt{a} Q^{t+1/2}, \sqrt{a} Q^{t+3/2}, \dots, \sqrt{a} Q^{M-1/2})  s_{\lambda} (\sqrt{a} \tilde{Q}^{1/2}, \sqrt{a} \tilde{Q}^{3/2}, \dots, \sqrt{a} \tilde{Q}^{N-1/2}) \\
            &= Z^{-1} a^{|\lambda|} Q^{(t+1/2)|\lambda|} \tilde{Q}^{|\lambda|/2} s_{\lambda} (1, Q, \dots, Q^{M-t-1}) s_{\lambda} (1, \tilde{Q}, \dots, \tilde{Q}^{N-1})
        \end{split}
    \end{equation}
    where $Z = \prod_{j=t+1}^M \prod_{i=1}^N (1-a Q^{i-1/2} \tilde{Q}^{j-1/2})^{-1}$ is the partition function, $s_{\lambda}$ are regular Schur polynomials ($s_\lambda = s_{\lambda/\emptyset}$), and we have used the homogeneity of the latter to get the second equation from the first, namely that $s_{\lambda} (c x_1, \dots, c x_n) = c^{|\lambda|} s_{\lambda} (x_1, \dots, x_n)$. 

    To finish, let us first notice that $\ell(\lambda) \leq M-t = L_t$ (we are looking at slice $-t$: $\lambda^{(-t)} = \lambda$). Specializing Schur polynomials in a geometric progression (the \emph{principal specialization}) is explicit~\cite[Ch.~I.3]{mac}:
    \begin{equation}
        s_{\lambda}(1, u, \dots, u^{n-1}) = \prod_{1 \leq i < j \leq n} \frac{u^{\lambda_i + L -i} - u^{\lambda_j + L - j}}{u^{L -i} - u^{L - j}}
    \end{equation} 
    where $L \geq \ell(\lambda)$ is arbitrary but big enough. We continue by taking $L=M$ above as global shift, and expanding the two Schur functions into two Vandermonde products of \emph{different} lengths $L_t$ and $N$ respectively. Finally, we recall that $l_i = \lambda_i + M - i$ and so $l_i = M-i$ for $i > L_t$. Thus the length $N$ Vandermonde can be rewritten as a length $L_t$ Vandermonde product times univariate factors of the form $(\tilde{Q}^{M-j}-\tilde{Q}^{l_i}) = \tilde{Q}^{M-j} (1-\tilde{Q}^{l_i - M + j})$ (for $N \geq j > L_t = M-t$). These factors give rise to the second Jacobi-like factor in the weight $w_d$. The first comes from writing $a^{|\lambda|} \propto \prod_{1 \leq i \leq L_t} a^{l_i}$ (and similarly for $\tilde{Q}^{|\lambda|/2}$ and $Q^{(t+1/2)|\lambda|}$) where we ignore gauge factors independent of the $l_i$'s. Elementary algebra finishes the proof.
\end{proof}

\begin{proof}[Proof of Thm.~\ref{thm:disc_corr}]

Okounkov and Reshetikhin~\cite{or03} have proven that the following extended point process
\begin{equation}
    \mathfrak{S} = \{ (t, \lambda^{(t)}_i - i + 1/2) \, :\, -M+1 \leq t \leq N - 1,\, i \geq 1 \}
\end{equation}
associated with the Schur process corresponding to $\Lambda$ in~\eqref{eq:schur_proc} is determinantal. More precisely, if we fix $-M+1 \leq t_1 < t_2 < \cdots < t_n \leq N-1$ and $k_i \in \Z+\tfrac12$ ($i=1 \dots n$), we have~\cite{or03}
\begin{equation}
    \P \left( k_i \in \{ \lambda^{(t)}_i - i + 1/2) \}, \forall \ 1 \leq i \leq n \right) = \det_{1 \leq i, j \leq n} \tilde{K}_d (t_i, k_i; t_j, k_j)
\end{equation}
where the time-extended kernel $\tilde{K}_d$ is given by:
\begin{equation} \label{eq:tilde_K_d}
    \tilde{K}_d (s, k; t, \ell) = \oint\limits_{|z| = 1 \pm \delta} \frac{dz}{2 \pi i z} \oint\limits_{|w| = 1 \mp \delta} \frac{dw}{2 \pi i w} \frac{F_d(s, z)}{F_d(t, w)} \frac{w^\ell}{z^k}\frac{\sqrt{zw}}{z-w} 
\end{equation}
where $F_d$ is as in the statement of Theorem~\ref{thm:disc_corr}; $\delta$ is small enough---see the same statement again and Figure~\ref{fig:K_d_cont} below; and for the contours we choose $|w| < |z|$ if $s \leq t$ and $|w| > |z|$ if $s  > t$. In the latter case of $s > t$ we first exchange the $w$ and $z$ contours picking up the residue which is the single contour integral in $K_d$. Shifting $k, \ell$ by $M+1/2$ and noticing that any slice $t$ has only at most the first $L_t$ particles in an ``excited state'' (equivalently $\ell(\lambda^{(t)}) \leq L_t$) completes the proof, giving the final formula for $K_d$ as a shifted $\tilde{K}_d$. 
\end{proof}

\begin{figure}[!h]
    \begin{center}
        \includegraphics[scale=0.5]{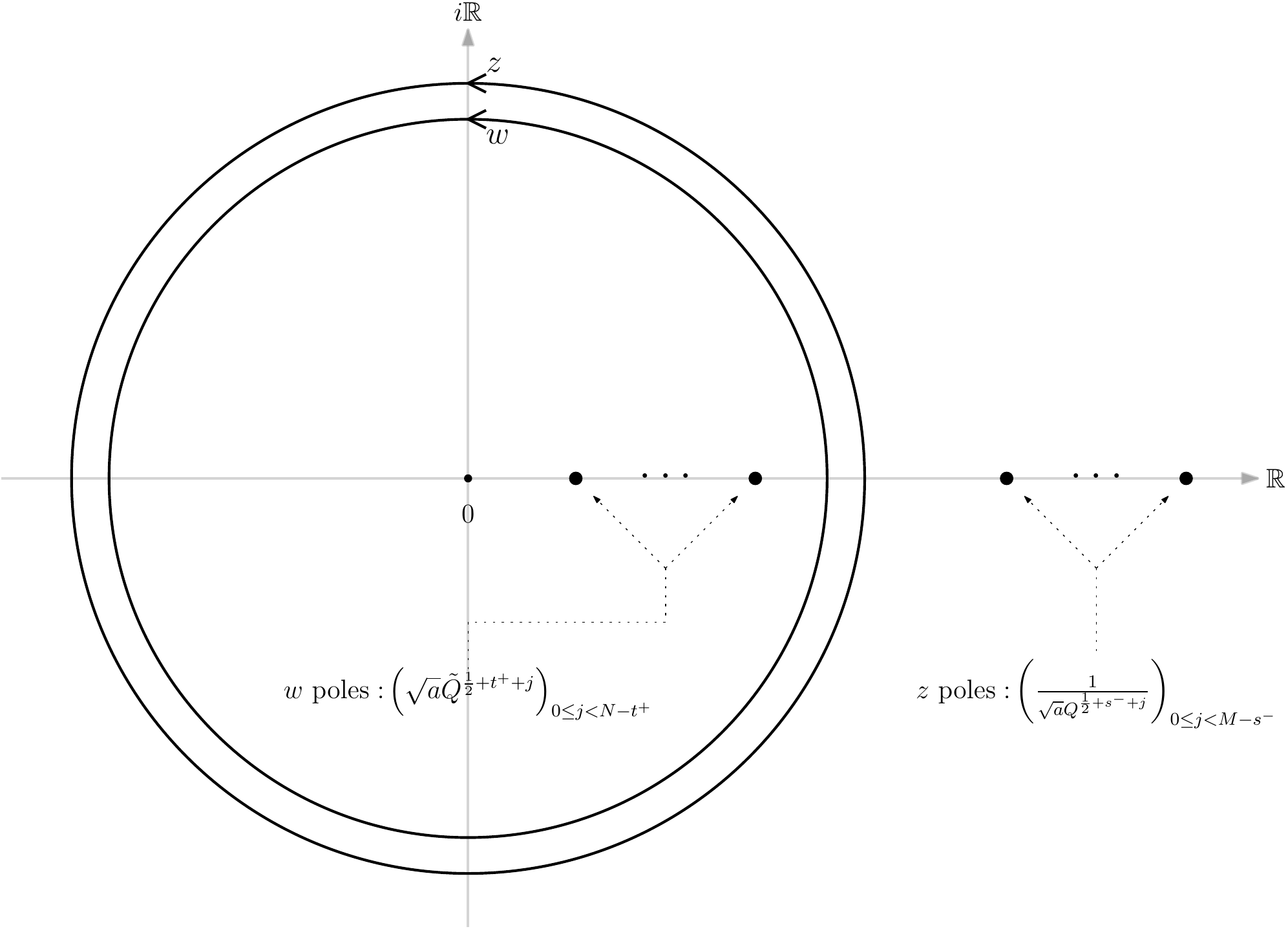}
    \end{center}
    \caption{The contours for $K_d$ in the case $s \leq t$. If $s>t$ the $w$ contour is reversed with the $z$ contour initially, and one exchanges them picking up a residual single contour integral. In the end they end up in the same order: $z$ on the outside and $w$ on the inside.}
    \label{fig:K_d_cont}
\end{figure}

\section{Proofs of Theorems~\ref{thm:cont_mb_dist}, \ref{thm:cont_corr}, and~\ref{thm:he_kernel}} \label{sec:cont_proofs}

The idea behind the proofs in this section is very simple: the first two proofs are $q \to 1-$ limits of Theorems~\ref{thm:disc_mb_dist} and~\ref{thm:disc_corr}; the proof of the hard-edge limit Theorem~\ref{thm:he_kernel} is a simple application of Stirling's approximation.

\begin{proof}[Proof of Theorem~\ref{thm:cont_mb_dist}] 

    Let us summarize the proof in words. As $q \to 1-$, the process $\Lambda$, via the associated shifted and truncated process $(l^{(t)})_t$, converges to the process ${\bf X} = (\x^{(t)})_t$ from the statement in the sense of weak convergence of finite dimensional distributions. Our proof is just a simple modification of the argument of Borodin--Gorin in~\cite[Section 2.3]{bg15}: instead of the Macdonald processes there we substitute our Schur process from~\eqref{eq:schur_proc}, and instead of the single principal specialization there we use two different principal specializations here (with steps $(Q, \tilde{Q}) = (q^\eta, q^\theta)$). 

    Let us sketch the argument. Recall the discrete setting: we start with $M,N$ fixed and a sequence of partitions $\Lambda$ as in~\eqref{eq:disc_interlacing} yielding the discrete process $(\lambda^{(t)})_{-M+1 \leq t \leq N-1}$ with $\ell(\lambda^{(t)}) \leq L_t$ by the interlacing constraints. Recall also the distribution on $\Lambda$ given in~\eqref{eq:pp_measure}. 
    
    Fixing $t$ and looking at the shifted ensemble $\{\lambda^{(t)}_i + M - i|i\geq 1\}$ we see that, due to the length constraints, we have that $\{\lambda^{(t)}_i + M - i|i > L_t\} = \{M-L_t-j|j \geq 1\}$ is fixed and deterministic throughout, so the only randomness is in the points of $\{\lambda^{(t)}_i + M - i|i \leq L_t\}$---this is the ensemble $l^{(t)}$. In the desired $q \to 1-$ limit:
\begin{equation}
    q = e^{-\epsilon}, \quad a = e^{-\alpha \epsilon}, \quad \lambda^{(t)}_i = -\frac{\log x_i^{(t)}}{\epsilon}, \quad \epsilon \to 0+
\end{equation}
with $\alpha \geq 0$ fixed a priori, the points in $q^{(l^{(t)})_t}$ (the non-trivial points in $\Lambda$ up to shift) go to points in a process ${\bf X}$ of interlacing vectors of real numbers in $(0, 1)$
\begin{equation}
    {\bf X} = \emptyset \prec \x^{(-M+1)} \prec \cdots \prec \x^{(-1)} \prec \x^{(0)} \succ \x^{(1)} \succ \cdot \succ \x^{(N-1)} \succ \emptyset
\end{equation}
where, from the discussion above, each $\x^{(t)}$ has exactly $L_t$ points. The discrete measure on $\Lambda$ from~\eqref{eq:pp_measure} becomes, via direct computations, the following measure:
\begin{equation}
    \begin{split}
    \P({\bf X}) d{\bf X} &= Z^{-1} \prod_{t = -M+1}^{-1} \prod_{i=1}^{L_t} \left( x^{(t)}_i \right)^{\eta-1} \cdot \prod_{i=1}^{M} \left( x^{(0)}_i \right)^{\alpha + \frac{\eta+\theta}{2} - 1} \cdot \prod_{t = 1}^{N-1} \prod_{i=1}^{L_t} \left( x^{(t)}_i \right)^{\theta-1} d{\bf X}\\
    &= Z^{-1} \prod_{t = -M+1}^{-1} \left| x^{(t)}_i \right|^{\eta-1} \cdot \left| x^{(0)}_i \right|^{\alpha + \frac{\eta+\theta}{2} - 1} \cdot \prod_{t = 1}^{N-1} \left| x^{(t)}_i \right|^{\theta-1}d{\bf X}
    \end{split}
\end{equation}
where $Z = \prod_{i=1}^M \prod_{j=1}^N (\alpha + \eta(i-\frac{1}{2}) + \theta(j-\frac{1}{2}))^{-1}$ is the partition function; $d {\bf X} = \prod_{-M+1 \leq t \leq N-1} \prod_{i=1}^{L_t} d x_i^{(t)}$; we have denoted $|\x| = \prod_i x_i$; and the $-1$ in each exponent comes from the differential $d \lambda^{(t)}_i \propto d x^{(t)}_i / x^{(t)}_i$. 

Moreover, the discrete marginals for $l^{(t)}$ also converge to the continuous marginals for $\x^{(t)}$ as announced. Clearly $Q^{l_i} \to x_i^{\eta}$; similarly for $\tilde{Q}$ and $\theta$; and likewise for $a^{l_i} \to x_i^\alpha$. What is less clear is what becomes of $w_d$ and for this we use the simple estimate 
\begin{align}
    \frac{(u^a f(u); u)_{\infty}}{(u^b f(u); u)_{\infty}} \to (1-f)^{b-a}
\end{align}
where $u \to 1-$, $f(u) \to f \in (0, 1)$, and to convert the finite length $q$-Pochhammer symbols to infinite ones we use $(x; u)_n = (x; u)_\infty/(x u^n; u)_\infty$ with $u \in \{Q, \tilde{Q}\}$ and $x$ as needed. This estimate is uniform over compact sets; see e.g.,~\cite[Lemma 2.4]{bg15} for an elementary proof. This shows that $w_d \to w_c$ as $q \to 1-$.

We make an important observation: we have been ignoring, for simplicity, powers of $\epsilon$ which need to be premultiplied to make, as $q \to 1-$, both sides of $\P(l^{(t)}) \to \P(\x^{(t)})$ finite. These powers can be easily recovered, and in fact are \emph{the same} as in~\cite[Section 2.3]{bg15}. This finishes the proof.
\end{proof}

\begin{proof}[Proof of Theorem~\ref{thm:cont_corr}] 
    Let us consider the desired $q \to 1-$ limit:
    \begin{equation}
        q = e^{-\epsilon}, \quad a = e^{-\alpha \epsilon}, \quad (k, \ell) = -\epsilon^{-1} (\log x, \log y), \quad \epsilon \to 0+ 
    \end{equation}
where $\alpha \geq 0$ is fixed and $x, y \in (0, 1)$. We show that, uniformly for $x, y$ in compact sets, we have
\begin{equation}
    \epsilon^{s-t-1} K_d (s, k; t, \ell) \to K_c (s, x; t, y).
\end{equation}

Ideas and computations here are standard and similar to those of e.~g.~Appendix C.3 in \cite{bfo20}, except we use different estimates since our ``action'' (integrand) is different. 

We can estimate the double contour integrand of $K_d$ using the following simple formula (uniform in $c$):
\begin{equation}
    \frac{(u^c; u)_n}{(1-u)^n} \to (c)_n, \quad u = e^{-r}, \quad r \to 0+
\end{equation}
with $u \in \{Q, \tilde{Q}\} = \{ q^\eta, q^\theta \}$ and $n$ ranging over the various lengths of $q$-Pochhammer symbols in $F_d(s, z)$ and $F_d(t, w)$. We first change variables as $(z, w) = (e^{\epsilon \zeta}, e^{\epsilon \omega})$, the contours becoming vertical lines close and parallel to $i \R$. Then we have the estimate
\begin{equation}
    \begin{split}
        \frac{F_d(s, z)}{F_d(t, w)} &\approx \frac{(\epsilon \theta)^{N-s^+} (\epsilon \eta)^{-(M-s^-)} }{ (\epsilon \theta)^{N-t^+} (\epsilon \eta)^{-(M-t^-)}  }    \cdot \frac{(\frac{\alpha}{2 \theta} + \frac{\zeta}{\theta} + s^+ + \frac{1}{2})_{N - s^+} } {(\frac{\alpha}{2 \eta} - \frac{\zeta}{\eta} + s^- + \frac{1}{2})_{M - s^-}}  \frac {(\frac{\alpha}{2 \eta} - \frac{\omega}{\eta} + t^- + \frac{1}{2})_{M - t^-}}  {(\frac{\alpha}{2 \theta} + \frac{\omega}{\theta} + t^+ + \frac{1}{2})_{N - t^+} } \\
        &\approx \epsilon^{t-s} \cdot \frac{F_c(s, \zeta)}{F_c(t, \omega)} \\
        \frac{w^\ell}{z^k} \frac{dz dw}{z-w} &\approx \epsilon \cdot \frac{x^\zeta}{y^\omega} \frac{d \zeta d \omega}{\zeta - \omega}
    \end{split}
\end{equation}
where the extra $\epsilon, \eta, \theta$ factors in the first estimate come from the left-overs $\frac{(1-\tilde{Q})^{N-s^+} (1-Q)^{-(M-s^-)} }{(1-\tilde{Q})^{N-t^+} (1-Q)^{-(M-t^-)}}$. The single contour integrand (for $s > t$) is estimated in a similar way (after all, it is just a residue of the double contour one). Finally let us mention that the factor $1/\sqrt{xy}$ in front of $K_c$ comes from the differential $dk \propto dx/x$---this introduces a factor $1/x$ in the kernel which we then conjugate to $1/\sqrt{xy}$ to make it more symmetric. 

By choosing the $(\zeta, \omega)$ contours (for the double contour integral) close to $i\R$ and bounded away from the poles (the arithmetic progressions $(\frac{\alpha}{2} + \frac{\eta}{2} + i s^- \eta)_{0 \leq i \leq M - s^- - 1}, (- \frac{\alpha}{2} - \frac{\theta}{2} - i t^+ \theta)_{0 \leq i \leq N - t^+ - 1}$ for $\zeta$ and $\omega$ respectively), we have that the integrand, multiplied by $\epsilon^{s-t-1}$, is bounded. The single contour integral of $K_d$ can be handled similarly. Dominated convergence allows us to take the limit inside the integral. It follows that $\epsilon^{s-t-1} K_d (s, k; t, \ell) \to K_{c}(s, x; t, y)$ uniformly for $x, y$ in compact sets as desired. Thus the process ${\bf X} = (\x^{(t)})_{-M+1 \leq t \leq N-1}$ is determinantal with kernel $K_c$ which is the $q \to 1-$ limit of the kernel $K_d$.
\end{proof}

\begin{rem}
    Another proof of this result, bypassing contour integrals altogether, is given in Appendix~\ref{sec:alt_c} Proposition~\ref{prop:K_c_alt_2}. The idea is that both kernels $K_d$ and $K_c$ are \emph{finite} sums of explicit simple residues, and then the limit $q \to 1-$ is immediate.
\end{rem}

\begin{proof}[Proof of Theorem~\ref{thm:he_kernel}]
    Notice we just need to handle the double contour integral part of $K_c$ in the limit $M, N\to \infty$; the single integral part (for $s > t$) is unchanged in the limit (the integrand is independent of $M, N$).

    Stirling's approximation of the $\Gamma$ function implies that $\Gamma(z+a)/\Gamma(z+b) \approx z^{a-b}, z \to \infty$ uniformly in $a, b$ over compact sets. Using the fact that $(x)_n = \Gamma(x+n)/\Gamma(x)$, let us write the $M, N$ dependent part of the ratio $F_c(s, \zeta)/F_c(t, \omega)$ and its asymptotics. The part that is $M, N$-independent becomes part of $K_{he}$. To wit, the announced dependence is
    \begin{equation}
        \begin{split}
            M \text{ dependence } &= \frac{\Gamma \left(\frac{\alpha}{2 \eta} - \frac{\omega}{\eta} + \frac{1}{2} + M\right)} {\Gamma\left(\frac{\alpha}{2 \eta} - \frac{\zeta}{\eta} + \frac{1}{2} + M\right)} \approx M^{\frac{1}{\eta} (\zeta-\omega)}, \quad M \to \infty, \\
            N \text{ dependence } &= \frac{\Gamma\left(\frac{\alpha}{2 \theta} + \frac{\zeta}{\theta} + \frac{1}{2} + N\right)}{\Gamma \left(\frac{\alpha}{2 \theta} + \frac{\omega}{\theta} + \frac{1}{2} + N \right)}  \approx N^{\frac{1}{\theta} (\zeta-\omega)}, \quad N \to \infty
        \end{split}
    \end{equation}
and we notice these powers are canceled after changing $(x, y) \mapsto \left( M^{-\frac{1}{\eta}} N^{-\frac{1}{\theta}}x, M^{-\frac{1}{\eta}} N^{-\frac{1}{\theta}} y\right)$ inside the $x^\zeta/y^{\omega}$ contribution. The overall prefactor $1/\sqrt{xy}$ contributes the overall scaling of $K_c$ by $M^{-\frac{1}{\eta}} N^{-\frac{1}{\theta}}$. To be able to interchange the limit with the (double) integral, we note that using the Hankel contours of Figures~\ref{fig:K_c_cont} and~\ref{fig:K_he_cont} with $(\zeta = \tau \pm i\epsilon, \omega = -\tau \pm i\epsilon), \ \tau \to \infty$ (and $\epsilon$ close to 0), we have exponential decay in $\zeta$ and $\omega$ (from $\tau$) from the $x^\zeta/y^\omega$ factor which kills any possible power (of $\zeta, \omega$) blow-up from the ratios of $\Gamma$ functions. This concludes the proof.
\end{proof}

\section{Random sampling and proofs of Theorems~\ref{thm:disc_lpp} and~\ref{thm:cont_lpp}} \label{sec:lpp_proofs}

Before we prove Theorems~\ref{thm:disc_lpp} and~\ref{thm:cont_lpp}, we discuss exact random sampling of: the discrete-space process $(l^{(t)})_{-M+1 \leq t \leq N-1}$ via the associated plane partition model; the associated plane partitions when $M=N=\infty$ (equivalently, the process $(\lambda^{(t)})_{t \in \Z}$); and the finite $M, N$ continuous-space process $(\x^{(t)})_{-M+1 \leq t \leq N-1}$. One reason is we find the algorithms simple to explain and intrinsically interesting. The second is that they are intimately related to the last passage percolation results we are striving to prove.

\subsection{Random sampling} \label{sec:sampling}

\subsubsection{The finite $M, N$ discrete case} \label{sec:disc_fin}

In what follows we explain how to sample plane partitions $\Lambda$ with $M \times N$ constrained base and distributed according to~\eqref{eq:pp_measure}. Equivalently, we explain how to sample the discrete Muttalib--Borodin process (MBP) $(l^{(t)})_{-M+1 \leq t \leq N-1}$. 

The randomness we start with is an integer matrix (rectangular grid) $(\omega_{I, J}^{\rm geo})_{1 \leq I \leq M, 1 \leq J \leq N}$ of independent random variables 
\begin{equation}
    \omega_{I, J}^{\rm geo} \sim {\rm Geom} (a Q^{M-I+\frac{1}{2}} \tilde{Q}^{N-J+\frac{1}{2}}).
\end{equation}
It is convenient to place them in the $(I, J)$ plane at positions $(I-1/2, J-1/2)$, as eventually we want to connect our procedure to last passage percolation results. Notice here that we have changed from the $(i, j)$ coordinates of Figures~\ref{fig:disc_lpp} and~\ref{fig:cont_lpp} (where the origin is top-right at $(i, j) = (1, 1)$) to $(I, J)$ (with origin bottom-left at $(I, J) = (0, 0)$). See Figure~\ref{fig:corner_growth} where the discrete setting is explained on the left.

\begin{figure}[!h]
    \begin{center}
        \includegraphics[scale=0.6]{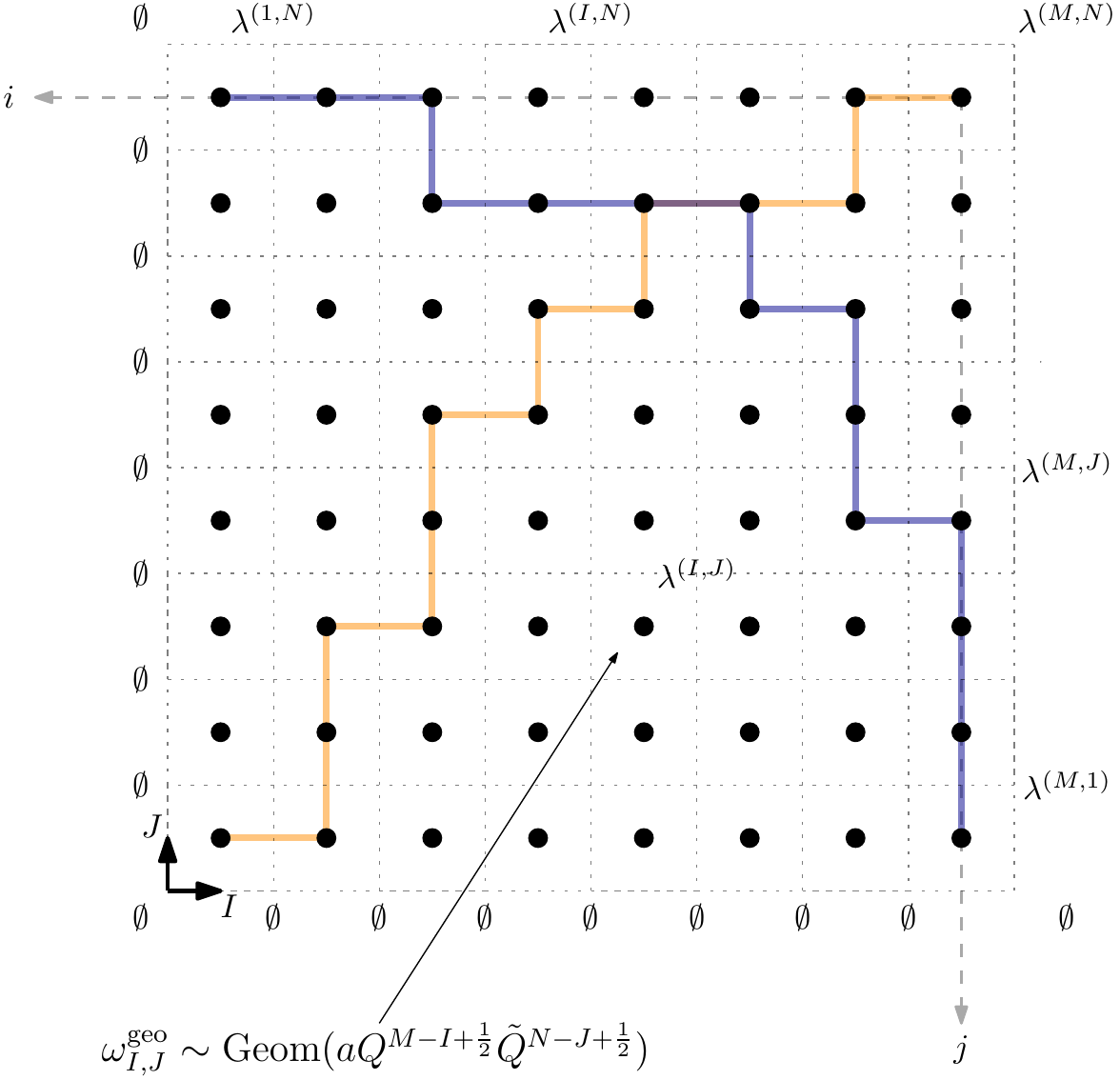} \quad \includegraphics[scale=0.6]{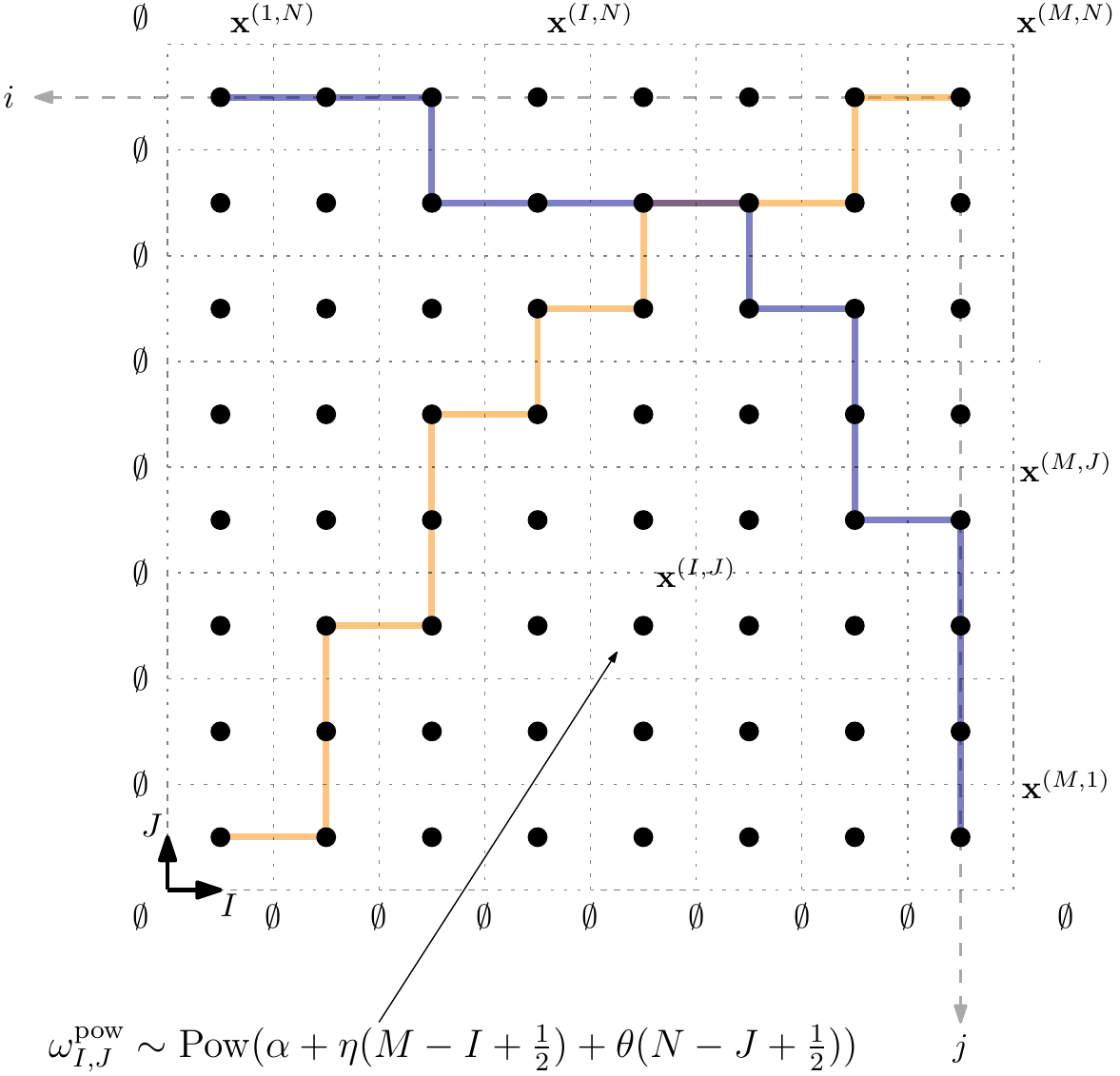}
    \end{center}
    \caption{The discrete finite $M, N$ field of random integers (left) and the continuous one (right). Dots represent random integers (left) and real numbers (right), independent of each other and distributed as indicated. Two possible paths/last passage times (over all that appear in the maximization/minimization of equations~\ref{eq:L_geo_fin} for left and~\ref{eq:L_pow_fin} for right respectively) are indicated in color: orange for the $L_1$'s and blue for the $L_2$'s. We also indicate the Fomin growth sampling algorithm schematically: we start with empty partitions (left) and vectors (right) on the bottom and left boundaries, and ``flip boxes'' using the local rules below constructing a fourth partition/vector at the North--East node based on the other three nodes and number in the middle and using the local rules described in this and the next sections. The output of the algorithm is then read on the top and right boundary, as an interlacing sequence of partitions/vectors.}
    \label{fig:corner_growth}
\end{figure}



We thus start our sampling procedure with the (random) non-negative integers $(\omega_{I,J}^{\rm geo})_{1 \leq I \leq M, 1 \leq J \leq N}$ sitting at the half-integer points $\left( I-1/2, J-1/2 \right)$ of $(\mathbb{N}+1/2)^2$ inside the rectangle $R:=[0, M] \times [0, N]$---see Figure~\ref{fig:corner_growth}. Random sampling takes place inductively, and we give three methods, each depending on the application of a particular ``local rule''. We next choose a local rule/algorithm 
\begin{equation} \label{eq:sampling_rules_disc}
    \mathcal{R}^{\rm geo} \in \{ rowRSK^{\rm geo}, colRSK^{\rm geo}, pushBlock^{\rm geo} \}
\end{equation} 
(to be defined below) which we apply inductively on the matrix $(\omega_{I,J}^{\rm geo})$. This rule is fixed once and for all throughout the whole sampling process.

On the integer points inside $R = [0, M] \times [0, N]$ we put the empty partition on the axes ($I = 0$ or $J = 0$). We then construct and place partitions $\lambda^{(I,J)}$ on the other lattice points---starting with $(1,1)$. Each $\lambda^{(I,J)}$ is defined inductively as the output of successive applications of the rule $\mathcal{R}^{\rm geo}$:
\begin{equation}
    \lambda^{(I,J)} = \begin{dcases}
        rRSK^{\rm geo} (\lambda^{(I-1,J)}, \lambda^{(I,J-1)}; \lambda^{(I-1,J-1)}; \omega_{I, J}^{\rm geo}), & \text{ if } \mathcal{R}^{\rm geo} = rowRSK^{\rm geo}, \\
        cRSK^{\rm geo} (\lambda^{(I-1,J)}, \lambda^{(I,J-1)}; \lambda^{(I-1,J-1)}; \omega_{I, J}^{\rm geo}), & \text{ if } \mathcal{R}^{\rm geo} = colRSK^{\rm geo}, \\
        pBlock^{\rm geo} (\lambda^{(I-1,J)}, \lambda^{(I,J-1)}; \omega_{I, J}^{\rm geo}; a Q^{M-I+\frac{1}{2}} \tilde{Q}^{N-J+\frac{1}{2}} ), & \text{ if } \mathcal{R}^{\rm geo} = pushBlock^{\rm geo}
    \end{dcases}
\end{equation}

The announced three local rules/algorithms giving rise to the three different ways of random sampling are given in the three displays below: \\ \ \\
\begin{tabular}{|l|l|}
    \hline
\begin{tabular}{@{}l@{}}
\noindent \textbf{Algorithm} $rowRSK^{\rm geo}$\\ \ \\
\noindent \textbf{Function} $rRSK^{\rm geo} (\alpha, \beta; \kappa; G)$\\ \ \\
\textbf{Input:} $\alpha, \beta; \kappa, G$ satisfying $\alpha \succ \kappa \prec \beta$\\
$\ell - 1 = \min(\ell(\alpha), \ell(\beta))$\\
$\nu_1 = \max(\alpha_1, \beta_1) + G$\\
\textbf{FOR} $s = 2, 3, \dots, \ell$\\
\indent  $\nu_s = \max(\alpha_s, \beta_s) + \min(\alpha_{s-1}, \beta_{s-1}) - \kappa_{s-1}$\\
\textbf{ENDFOR}\\
\textbf{Output:} $\nu$ satisfying $\alpha \prec \nu \succ \beta$\\
\ \\
\
\end{tabular}
& 
\begin{tabular}{@{}l@{}}
\noindent \textbf{Algorithm} $colRSK^{\rm geo}$\\ \ \\
\noindent \textbf{Function} $cRSK^{\rm geo} (\alpha, \beta; \kappa; G)$\\ \ \\
\textbf{Input:} $\alpha, \beta; \kappa, G$ satisfying $\alpha \succ \kappa \prec \beta$ \\
$\ell - 1 = \min(\ell(\alpha), \ell(\beta))$ \\
$G_{\ell} = G$ \\
\textbf{FOR} $s = \ell, \ell - 1, \dots, 1$  \\
\indent  $\nu_s = \min(\max(\alpha_s, \beta_s) + G_s, \kappa_{s-1})$ \\
\indent  $G_{s-1} = G_s - \min(G_s, \kappa_{s-1} - \max(\alpha_s, \beta_s))$ \\
\indent \indent \indent \indent $+ \min(\alpha_{s-1}, \beta_{s-1}) - \kappa_{s-1}$ \\
\textbf{ENDFOR} \\
\textbf{Output:} $\nu$ satisfying $\alpha \prec \nu \succ \beta$
\end{tabular}\\
    \hline
\end{tabular}
\\ \ \\
and finally 
\\ \ \\
\begin{tabular}{|l|}
    \hline
\begin{tabular}{@{}l@{}}
    \noindent \textbf{Algorithm} $pushBlock^{\rm geo}$\\ \ \\
    \noindent \textbf{Function} $pBlock^{\rm geo} (\alpha, \beta; G; p)$\\ \ \\
    \textbf{Input:} $\alpha, \beta$ \\
    $\ell - 1 = \min(\ell(\alpha), \ell(\beta))$\\
    $\nu_1 = \max(\alpha_1, \beta_1) + G $\\
    \textbf{FOR} $s = 2, 3, \dots, \ell$\\
    \indent  $\nu_s = \max(\alpha_s, \beta_s) + {\rm Geom}_{ \min(\alpha_{s-1}, \beta_{s-1}) - \max(\alpha_s, \beta_s) } ( p )$\\
    \textbf{ENDFOR}\\
    \textbf{Output:} $\nu$ satisfying $\alpha \prec \nu \succ \beta$\\
\end{tabular}\\
\hline
\end{tabular}
\\ \ \\
\noindent (where in the last display we call $X \in \{ 0, 1, \dots, n \}$ a \emph{truncated geometric random variable} ${\rm Geom}_n(q)$  if $\P(X=k) = \frac{(1-q)q^k}{1-q^{n+1}} $). 

The whole algorithm is simple to explain in words: once $\mathcal{R}^{\rm geo}$ is chosen, we sample $\lambda^{(I, J)}$ as soon as $\lambda^{(I-1, J)}$ and $\lambda^{(I, J-1)}$ (and possibly $\lambda^{(I-1, J-1)}$) become available. The boundary conditions make the initial step possible. The order of the individual steps is not important. For the $rowRSK^{\rm geo}$ and $colRSK^{\rm geo}$ rules, the sole randomness is in $\omega_{I, J}^{\rm geo}$. For $pushBlock^{\rm geo}$ there is more randomness in each step. Note that $\ell(\lambda^{(I, J)}) \leq \min(I, J)$ by construction.

The output of the algorithm is a random sequence of $M+N+1$ interlacing partitions
\begin{equation} \label{eq:interlacing_output}
    \begin{split}
        \Lambda &= \emptyset \prec \lambda^{(1, N)} \prec \dots \prec \lambda^{(M-1, N)} \prec \lambda^{(M, N)} \succ \lambda^{(M, N-1)} \succ \dots \succ \lambda^{(M, 1)} \succ \emptyset \\
        &= \emptyset \prec \lambda^{(-M+1)} \prec \dots \prec \lambda^{(-1)} \prec \lambda^{(0)} \succ \lambda^{(1)} \succ \dots \succ \lambda^{(N-1)} \succ \emptyset
    \end{split}
\end{equation}
where the second line is just a rewriting of the first to make the notation correspond to~\eqref{eq:disc_interlacing}.

\begin{rem} \label{rem:disc_bij}
    Local rules $rowRSK^{\rm geo}$ and $colRSK^{\rm geo}$ are bijections. In fact they are Fomin growth diagram~\cite{fom86, fom95} reinterpretations of the classical row insertion Robinson--Schensted--Knuth correspondence~\cite{knu70} and column insertion Burge correspondence~\cite{bur74} respectively\footnote{The local rule $colRSK^{\rm geo}$ is called ``Burge'' in~\cite{bcgr20, boz20}. In~\cite{bcgr20} it is used to obtain a result of similar flavor to ours, that certain diagonal last passage times are the same in distribution as certain anti-diagonal ones---see Theorem~\ref{thm:greene_disc_fin}.}. More precisely, fixing partitions $\alpha$ and $\beta$, both rules provide bijections between the two sets 
    \begin{equation}
        \{\kappa \text{ a partition}: \alpha \succ \kappa \prec \beta \} \times \{G: G \in \mathbb{N} \} \longleftrightarrow \{\nu \text{ a partition}: \alpha \prec \nu \succ \beta \}
    \end{equation}
    satisfying the condition $|\kappa| + |\nu| = |\alpha| + |\beta| + G$. There are many other bijections satisfying the same conditions, but these two also satisfy Greene-type theorems of interest in the sequel. Finally, $pushBlock^{\rm geo}$, while not a bijection per se, can be seen as a randomized bijection. See~\cite{pm17, bp16} for more examples and for a short explanation of the terminology. 
\end{rem}

The following theorem states that the above procedure samples exactly from the distribution we want. Let us attribute it properly. The $pushBlock^{\rm geo}$ sampling and dynamics originally appears in Borodin's work~\cite{bor11}. For $rowRSK^{\rm geo}$ one can check~\cite{bbbccv18} and references therein. For $colRSK^{\rm geo}$ see the Appendix of~\cite{b18} and references therein. For $rowRSK^{\rm geo}$  the result is also in~\cite{fr05}. 

\begin{thm}
    For any of the three sampling local rules $\mathcal{R}^{\rm geo}$ of~\eqref{eq:sampling_rules_disc}, the output of the algorithm as rewritten in the second line of~\eqref{eq:interlacing_output} is a random plane partition with base inside an $M \times N$ rectangle distributed according to~\eqref{eq:pp_measure}. Equivalently, the algorithm exactly samples the associated discrete MBP (point process) $(l^{(t)})_{-M+1 \leq t \leq N-1}$ of Section~\ref{sec:pp}. 
\end{thm}

\begin{rem}
    Assuming $O(1)$ complexity for sampling the necessary geometric random variables, each of the three algorithms takes $O(M^2 N)$ operations.
\end{rem}

Let us now record the connection to directed last passage percolation. Consider the following two random variables: 
\begin{equation} \label{eq:L_geo_fin}
    \begin{split}
        L_1^{\rm geo} = \max\limits_{\pi} \left[ \sum\limits_{(I-\frac{1}{2}, J-\frac{1}{2}) \in \pi} \omega_{I,J}^{\rm geo} \right], \quad L_2^{\rm geo} = \max\limits_{\varpi} \left[ \sum\limits_{(I-\frac{1}{2}, J-\frac{1}{2}) \in \varpi} \omega_{I,J}^{\rm geo} \right] 
    \end{split}
\end{equation}
where the sum in $L_1$ is taken over all up-right lattice paths $\pi$ from $(1/2, 1/2)$ to $(M-1/2, N-1/2)$; and the sum in $L_2$ is taken over all down-right lattice paths $\varpi$ from $(1/2, N-1/2)$ to $(M-1/2, 1/2)$.

The connection to last passage percolation is contained in the following result. See the Appendix of~\cite{b18} for some history. The theorem is attributed to Greene~\cite{gre74} for $rowRSK^{\rm geo}$ (and by extension for $pushBlock^{\rm geo}$); and to Krattenthaler~\cite{kra06} for $colRSK^{\rm geo}$. 

\begin{thm}[Greene--Krattenthaler theorem, finite discrete setting] \label{thm:greene_disc_fin}
    Using either of the algorithms/local rules $rowRSK^{\rm geo}$ or $pushBlock^{\rm geo}$ for sampling, we have $\lambda^{(0)}_1 = L_1^{\rm geo}$. Using $colRSK^{\rm geo}$ as local rule, we have $\lambda^{(0)}_1 = L_2^{\rm geo}$. In particular $L_1^{\rm geo} = L_2^{\rm geo}$ in distribution. 
\end{thm}

\subsubsection{The $M = N = \infty$ discrete case} \label{sec:disc_inf}

It is possible to sample plane partitions $\Lambda$ distributed according to~\eqref{eq:pp_measure} but with unconstrained base. In other words, it is possible to sample such plane partitions when $M = N = \infty$. The idea is that given weights on the $(i,j)_{i,j\geq 1}$ lattice which are ${\rm Geom} (a Q^{i - \frac{1}{2}} \tilde{Q}^{j - \frac{1}{2}})$---notice we are using the $(i,j)$ coordinates of Figure~\ref{fig:disc_lpp}, only finitely many of them will be non-zero by Borel--Cantelli. We then choose $M, N$ big enough so that outside an $M\times N$ rectangle these weights are almost surely zero. We switch coordinates to $(I, J)$ (and the origin to the lower-left corner as in Figure~\ref{fig:corner_growth}) and apply any of the three algorithms described in the previous section. How big $M$ and $N$ need to be given explicit $a, Q, \tilde{Q}$ is discussed in Section 5 of~\cite{bbbccv18} where the authors also sketch the algorithm itself using $rowRSK^{\rm geo}$ (but any of the other two local rules can be used); the $a_i$ and $b_j$ \emph{Schur parameters} in op.~cit.~ are equal to $(a_i = \sqrt{a} Q^{i-1/2}, b_j = \sqrt{a} \tilde{Q}^{j-1/2})_{i, j \geq 1}$ in our case. Technically speaking, the sampling is not exact in this way, as an approximation for $M$ and $N$ has to be made. 

What is nonetheless important for us is that the above Greene--Kratthenthaler Theorem~\ref{thm:greene_disc_fin} still holds. Using the coordinates $(i, j)$ of Figure~\ref{fig:disc_lpp} and Section~\ref{sec:disc_lpp}, let 
\begin{equation}
    \begin{split}
        L_1^{\rm geo} = \max\limits_{\pi} \left[ \sum\limits_{(i, j) \in \pi} \omega_{i, j}^{\rm geo} \right], \quad L_2^{\rm geo} = \max\limits_{\varpi} \left[ \sum\limits_{(i, j) \in \varpi} \omega_{i, j}^{\rm geo} \right] 
    \end{split}
\end{equation}
where the sum in $L_1^{\rm geo}$ is taken over all down-left lattice paths $\pi$ from $(1, 1)$ to $(\infty, \infty)$; and the sum in $L_2^{\rm geo}$ is taken over all down-right lattice paths $\varpi$ from $(\infty, 1)$ to $(1, \infty)$.

We reformulate the analogue of Theorem~\ref{thm:greene_disc_fin} in a setting suitable for our needs. 

\begin{thm}[Greene--Krattenthaler theorem, infinite discrete setting] \label{thm:greene_disc_inf}
    We have $L_1 = L_2$ in distribution and both are equal to $\lambda_1$, the first part of a partition $\lambda$ distributed according to the Schur measure
    \begin{equation} \label{eq:schur_inf}
        \begin{split}
        \P(\lambda) &= Z^{-1} s_{\lambda} (\sqrt{a} Q^{1/2}, \sqrt{a} Q^{3/2}, \sqrt{a} Q^{5/2}, \dots) s_{\lambda} (\sqrt{a} \tilde{Q}^{1/2}, \sqrt{a} \tilde{Q}^{3/2}, \sqrt{a} \tilde{Q}^{5/2}, \dots) \\
        &= Z^{-1} \left[ a (Q \tilde{Q})^{1/2} \right]^{|\lambda|} s_{\lambda} (1, Q, Q^2, \dots) s_{\lambda} (1, \tilde{Q}, \tilde{Q}^{2}, \dots)
        \end{split}
    \end{equation}
    with $s_\lambda$ denoting Schur functions (now in infinitely many variables) and $Z = \prod_{i, j \geq 1} (1-a Q^{i-1/2} \tilde{Q}^{j-1/2})^{-1}$. 
\end{thm}

\subsubsection{The finite $M, N$ continuous case} \label{sec:cont_fin}

In what follows we explain how to sample the MBP (point process) $(\x^{(t)})_{-M+1 \leq t \leq N-1}$ of Section~\ref{sec:cont_mbe}. Throughout we sample vectors $\x$ of finite length and increasing entries containing random numbers strictly between 0 and 1. It is convenient to denote by $\emptyset$ the empty vector containing no such numbers, which also conveniently can be thought of as the vector $\emptyset = (1,1,1,1,1,\dots)$ containing infinitely many 1's. 

The setting is similar to that of the discrete finite $M, N$ case, and all arguments there apply here as well with modifications we now explain.

The randomness we start with is a matrix (rectangular grid) $(\omega_{I, J}^{\rm geo})_{1 \leq I \leq M, 1 \leq J \leq N}$ of independent power random variables in $[0, 1]$
\begin{equation}
    \omega_{I, J}^{\rm pow} \sim {\rm Pow} \left(\alpha + \eta(M-I+\tfrac12) + \theta (N-J+\tfrac12)\right)
\end{equation}
which we conveniently place in the $(I, J)$ plane at positions $(I-1/2, J-1/2)$ as in Figure~\ref{fig:corner_growth} (right). Notice these are $q \to 1-$ limits of the aforementioned geometric random variables. To wit, if $\omega_{I, J}^{\rm geo} = {\rm Geom} (a Q^{M-I+\frac{1}{2}} \tilde{Q}^{N-J+\frac{1}{2}})$, then in the limit 
\begin{equation}
    q = e^{-\epsilon}, \quad a = e^{-\alpha \epsilon}, \quad \epsilon \to 0+
\end{equation}
with (recall) $Q = q^\eta, \tilde{Q}=q^\theta$ we have
\begin{equation}
    e^{ -\epsilon \omega_{I, J}^{\rm geo}} \to \omega_{I, J}^{\rm pow} \sim {\rm Pow} \left(\alpha + \eta(M-I+\tfrac12) + \theta (N-J+\tfrac12)\right).
\end{equation}

We proceed exactly as before. We start with the (random) real numbers $(\omega_{I,J}^{\rm geo})_{1 \leq I \leq M, 1 \leq J \leq N}$ sitting at the half-integer points $\left( I-1/2, J-1/2 \right)$ of $(\mathbb{N}+1/2)^2$ inside the rectangle $R:=[0, M] \times [0, N]$---see Figure~\ref{fig:corner_growth}. The sampling is inductive, we again give three methods based on choosing a local rule $\mathcal{R}^{\rm pow}$ out of three: 
\begin{equation} \label{eq:sampling_rules_cont}
    \mathcal{R}^{\rm pow} \in \{ rowRSK^{\rm pow}, colRSK^{\rm pow}, pushBlock^{\rm pow} \}.
\end{equation} 
Each individual rule is defined below.

On the integer points inside $R = [0, M] \times [0, N]$ we put the empty vector $\emptyset$ on the axes ($I = 0$ or $J = 0$). We then construct and place vectors $\x^{(I,J)}$ on the other lattice points---starting with $(1,1)$. Each $\x^{(I,J)}$ is defined inductively as the output of successive applications of the rule $\mathcal{R}^{\rm pow}$:
\begin{equation}
    \x^{(I,J)} = \begin{dcases}
        rRSK^{\rm pow} (\x^{(I-1,J)}, \x^{(I,J-1)}; \x^{(I-1,J-1)}; \omega_{I, J}^{\rm pow}), & \mathcal{R}^{\rm pow} = rowRSK^{\rm pow}, \\
        cRSK^{\rm pow} (\x^{(I-1,J)}, \x^{(I,J-1)}; \x^{(I-1,J-1)}; \omega_{I, J}^{\rm pow}), & \mathcal{R}^{\rm pow} = colRSK^{\rm pow}, \\
        pBlock^{\rm pow} (\x^{(I-1,J)}, \x^{(I,J-1)}; \omega_{I, J}^{\rm pow}; \alpha + \eta(M-I+\tfrac12) + \theta(N-J+\tfrac12) ), & \mathcal{R}^{\rm pow} = pushBlock^{\rm pow}
    \end{dcases}
\end{equation}

The announced three local rules/algorithms replacing the three discrete ones above are below: \\ \ \\
\begin{tabular}{|l|l|}
    \hline
\begin{tabular}{@{}l@{}}
\noindent \textbf{Algorithm} $rowRSK^{\rm pow}$\\ \ \\
\noindent \textbf{Function} $rRSK^{\rm pow} (\a, \b; \u; g)$\\ \ \\
\textbf{Input:} $\a, \b; \u, g$ satisfying $\a \succ \u \prec \b$\\
$\ell - 1 = \min(\ell(\a), \ell(\b))$\\
$\v_1 = g \min(\a_1, \b_1)$\\
\textbf{FOR} $s = 2, 3, \dots, \ell$\\
\indent  $\v_s = \frac {\min(\a_s, \b_s) \max(\a_{s-1}, \b_{s-1}) } { \u_{s-1} }$\\
\textbf{ENDFOR}\\
\textbf{Output:} $\v$ satisfying $\a \prec \v \succ \beta$\\
\
\end{tabular}
& 
\begin{tabular}{@{}l@{}}
\noindent \textbf{Algorithm} $colRSK^{\rm pow}$\\ \ \\
\noindent \textbf{Function} $cRSK^{\rm pow} (\a, \b; \u; g)$\\ \ \\
\textbf{Input:} $\a, \b; \u, g$ satisfying $\a \succ \u \prec \b$\\
$\ell - 1 = \min(\ell(\a), \ell(\b))$\\
$g_{\ell} = g$ \\
\textbf{FOR} $s = \ell, \ell - 1, \dots, 1$  \\
\indent  $\v_s = \max(g_s \min(\a_s, \b_s), \u_{s-1})$ \\
\indent  $g_{s-1} = \frac{ g_s \max(\a_{s-1}, \b_{s-1}) }  {\u_{s-1}  \max\left( g_s, \frac{\u_{s-1}} {\min(\a_s, \b_s)} \right)  }$ \\
\textbf{ENDFOR} \\
\textbf{Output:} $\v$ satisfying $\a \prec \v \succ \b$
\end{tabular}\\
    \hline
\end{tabular}
\\ \ \\
and finally 
\\ \ \\
\begin{tabular}{|l|}
    \hline
\begin{tabular}{@{}l@{}}
    \noindent \textbf{Algorithm} $pushBlock^{\rm pow}$\\ \ \\
    \noindent \textbf{Function} $pBlock^{\rm pow} (\a, \b; g; \gamma)$\\ \ \\
    \textbf{Input:} $\a, \b$ \\
    $\ell - 1 = \min(\ell(\alpha), \ell(\beta))$\\
    $\v_1 = g \min(\a_1, \b_1)$\\
    \textbf{FOR} $s = 2, 3, \dots, \ell$\\
    \indent  $\v_s = \min(\a_s, \b_s)  {\rm Pow}_{ \frac{\max(\a_{s-1}, \b_{s-1})} {\min(\a_s, \b_s) } } ( \gamma )$\\
    \textbf{ENDFOR}\\
    \textbf{Output:} $\v$ satisfying $\a \prec \v \succ \b$\\
\end{tabular}\\
\hline
\end{tabular}
\\ \ \\
\noindent (where in the last display we call $Y \in [A, 1]$ a \emph{truncated power random variable} ${\rm Pow}_A( \gamma )$  if $\P(Y \in dy) = \Id_{[A, 1]} \frac{\gamma y^{\gamma-1}} {1-A^{\gamma}}   $). 

To summarize: once $\mathcal{R}^{\rm geo}$ is chosen, we sample $\x^{(I, J)}$ as soon as $\x^{(I-1, J)}$ and $\x^{(I, J-1)}$ (and possibly $\x^{(I-1, J-1)}$) become available. The boundary conditions make the initial step possible. The order of the individual steps is not important for the sampling procedure, but note that now there is a natural order if one thinks of the $\omega_{(I, J)}^{\rm pow}$'s as times. For the $rowRSK^{\rm pow}$ and $colRSK^{\rm pow}$ rules, the sole randomness is in $\omega_{I, J}^{\rm pow}$. For $pushBlock^{\rm pow}$ there is more randomness in each step. Note that $\ell(\x^{(I, J)}) = \min(I, J)$ by construction; unlike in the geometric case, we now have equality. This makes the algorithm even simpler to implement than the geometric one.

The output of the algorithm is a random sequence of $M+N+1$ interlacing vectors
\begin{equation} \label{eq:interlacing_output_cont}
    \begin{split}
        {\bf X} &= \emptyset \prec \x^{(1, N)} \prec \dots \prec \x^{(M-1, N)} \prec \x^{(M, N)} \succ \x^{(M, N-1)} \succ \dots \succ \x^{(M, 1)} \succ \emptyset \\
        &= \emptyset \prec \x^{(-M+1)} \prec \dots \prec \x^{(-1)} \prec \x^{(0)} \succ \x^{(1)} \succ \dots \succ \x^{(N-1)} \succ \emptyset
    \end{split}
\end{equation}
where the second line is just a rewriting of the first to make the notation correspond to~\eqref{eq:cont_interlacing_2}.

\begin{rem} \label{rem:cont_bij}
    Similarly to Remark~\ref{rem:disc_bij}, we have that local rules $rowRSK^{\rm pow}$ and $colRSK^{\rm pow}$ are bijections. This means a lot less from a combinatorial point of view since we are dealing with vectors of real numbers, but nonetheless we can make things explicit. Fixing vectors $\a$ and $\b$, both rules provide bijections between the following two sets of vectors (by \emph{vector} we implicitly mean a vector of finitely many increasing numbers strictly between 0 and 1 and padded with infinitely many 1's at the end)
    \begin{equation}
        \{\u \text{ a vector}: \a \succ \u \prec \b \} \times \{g: g \in (0, 1) \} \longleftrightarrow \{\v \text{ a vector}: \a \prec \v \succ \b \}
    \end{equation}
    satisfying the condition $|\u| \cdot |\v| = |\a| \cdot |\b| \cdot g$ where we have denoted $|\x| = \prod_i x_i$ (this is well defined as neither component is zero and only finitely many are not equal to 1). As before, there are many other bijections satisfying the same conditions, but these two are of special interest to us as can be seen below in connection to last passage percolation models. Finally, $pushBlock^{\rm pow}$ can be seen as a randomized bijection. 
\end{rem}

The following theorem states that the above procedure samples exactly the process $(\x^{(t)})_t$ we are after. The $rowRSK^{\rm pow}$ dynamics is already discussed at length in~\cite{fr05}, except in ``logarithmic variables'', meaning in the variables $y_i^{(t)} = -\log x_i^{(t)}$.

\begin{thm} \label{thm:cont_sampling}
    For any of the three sampling local rules $\mathcal{R}^{\rm pow}$ of~\eqref{eq:sampling_rules_cont}, the output of the algorithm as rewritten in the second line of~\eqref{eq:interlacing_output_cont} is the process $(\x^{(t)})_{-M+1 \leq t \leq N-1}$ of Theorem~\ref{thm:cont_mb_dist}. 
\end{thm}

\begin{proof}
    The proof relies on the simple observation that the limit $q \to 1-$ of the sampling algorithm described in the finite $M, N$ discrete case above makes sense. Let us recall that in the limit
    \begin{equation}
        q = e^{-\epsilon}, \quad a = e^{-\alpha \epsilon}, \quad \lambda^{(t)}_i = -\epsilon^{-1} \log x_i^{(t)}, \quad \epsilon \to 0+
    \end{equation}
    (with $(Q, \tilde{Q}) = (q^\eta, q^\theta)$), the process $(l^{(t)})_{-M+1 \leq t \leq N-1}$ converges to the process $(\x^{(t)})_{-M+1 \leq t \leq N-1}$ (recall that $l^{(t)}_i = \lambda^{(t)}_i - M + i$ for all $1 \leq i \leq L_t$). Moreover one has, as explained above, that $\exp( -\epsilon \omega_{I, J}^{\rm geo}) \to \omega_{I, J}^{\rm pow}$, i.~e.~the geometric random variables become power random variables. 
    
    The local rules are also well-behaved under the limit, notably due to the logarithm function present in going from discrete to continuous. For example, the rule for constructing the first part of the partition $\nu$ inside $rowRSK^{\rm pow}$, which reads 
    \begin{equation}
        \nu_1 = \max(\alpha_1, \beta_1) + G
    \end{equation}  
    becomes, under the substitution $(\nu_1, \alpha_1, \beta_1, G) = -\epsilon^{-1} (\log \v_1, \log \a_1, \log \b_1, \log g)$, as follows:
    \begin{equation}
        -\frac{\log \v_1}{\epsilon} = \max\left(-\frac{\log \a_1}{\epsilon}, -\frac{\log \b_1}{\epsilon} \right) - \frac{\log g}{\epsilon}.
    \end{equation}
    Clearing out the $\epsilon$ and the signs, using the fact that $\max(-x, -y) = -\min(x, y)$, that $\min(\log c, \log d) = \log \min (c, d)$ (as $\log$ is increasing, and similarly for $\max$ replacing $\min$), and finally exponentiating everything in the end we end up with the announced rule given in $rowRSK^{\rm pow}$ for the first part of the vector $\v$:
    \begin{equation}
        \v_1 = g \min(\a_1, \b_1).
    \end{equation} 
    Similarly simple computations show that every one of the three local rules $\mathcal{R}^{\rm geo}$ ``converges'' to the corresponding rule $\mathcal{R}^{\rm pow}$. Note that for the $pushBlock$ rules, we also need to use the (true) observation that the truncated geometric random variables used in the discrete case converge, in the same sense as the non-truncated ones, to the truncated power variables of the continuous setting.

    To summarize the proof, we have three algorithms for exact random sampling of the discrete process $(l^{(t)})_t$ (via the partition process $(\lambda^{(t)})_{t}$) from the randomness $(\omega_{I, J}^{\rm geo})_{I, J}$ where we omit ranges for brevity. In the limit $q \to 1-$ the aforementioned discrete process converges to $(\x^{(t)})_t$, the randomness converges to that of the matrix $(\omega_{I, J}^{\rm pow})_{I, J}$, and all local rules $\mathcal{R}^{\rm geo}$ have well-defined $q \to 1-$ analogues in the corresponding $\mathcal{R}^{\rm pow}$. It follows that the usage of the latter rules samples exactly the process $(\x^{(t)})_t$ starting from the randomness $(\omega_{I, J}^{\rm pow})_{I, J}$.
\end{proof}

Let us now record the connection to the directed last passage percolation model of Section~\ref{sec:cont_lpp}. Consider the following two random variables defined in sec.~cit.: 
\begin{equation} \label{eq:L_pow_fin}
    \begin{split}
        L_1^{\rm pow} = \min\limits_{\pi} \left[ \prod\limits_{(I-\frac{1}{2}, J-\frac{1}{2}) \in \pi} \omega_{I,J}^{\rm pow} \right], \quad L_2^{\rm pow} = \min\limits_{\varpi} \left[ \prod\limits_{(I-\frac{1}{2}, J-\frac{1}{2}) \in \varpi} \omega_{I,J}^{\rm pow} \right] 
    \end{split}
\end{equation}
where the product in $L_1$ is taken over all up-right lattice paths $\pi$ from $(1/2, 1/2)$ to $(M-1/2, N-1/2)$; and the product in $L_2$ is taken over all down-right lattice paths $\varpi$ from $(1/2, N-1/2)$ to $(M-1/2, 1/2)$.

We have the following continuous analogue of Theorem~\ref{thm:greene_disc_fin}.

\begin{thm}[Greene--Krattenthaler theorem, finite continuous setting]~\label{thm:greene_cont_fin}
    Using either $rowRSK^{\rm pow}$ or $pushBlock^{\rm pow}$ as local rules, we have $\x^{(0)}_1 = L_1^{\rm pow}$. Using $colRSK^{\rm pow}$ as local rule, we have $\x^{(0)}_1 = L_2^{\rm pow}$. In particular, $L_1^{\rm pow} = L_2^{\rm pow}$ in distribution. 
\end{thm}

\begin{proof}
    The proof is again a simple $q \to 1-$ limit of the result in the discrete finite $M,N$ version, namely Theorem~\ref{thm:greene_disc_fin}. Recall that as $q = e^{-\epsilon}$ for $\epsilon \to 0+$ we have $\exp( -\epsilon \omega_{I, J}^{\rm geo}) \to \omega_{I, J}^{\rm pow}$. Then 
    \begin{equation}
        \max \sum \omega_{I, J}^{\rm geo} = \max \left( - \sum \epsilon^{-1} \log \omega_{I, J}^{\rm pow} \right) = - \epsilon^{-1} \log \min \left( \prod \omega_{I, J}^{\rm pow} \right) 
    \end{equation} 
    which shows that $\exp(- \epsilon L_i^{\rm geo} ) \to L_i^{\rm pow}$ (for $i = 1, 2$). Together with the fact that $\exp(-\epsilon \lambda^{(0)}_1) \to x^{(0)}_1$ and the discrete finite $M, N$ Greene--Krattenthaler Theorem~\ref{thm:greene_disc_fin}, this finishes the proof.
\end{proof}


\subsection{Proofs of Theorems~\ref{thm:disc_lpp} and~\ref{thm:cont_lpp} }

We begin by proving the continuous last passage percolation result as it is the easiest. 

\begin{proof}[Proof of Theorem~\ref{thm:cont_lpp}]
    By Theorem~\ref{thm:greene_cont_fin} (with $M=N$) we have that $L^{\rm pow}_1 = L^{\rm pow}_2 = x^{(0)}_1$ in distribution, and then Theorem~\ref{thm:gap_prob_cont} gives the desired result.
\end{proof}

\begin{rem}
    Let us remark here that the hard-edge gap probability result of Theorem~\ref{thm:gap_prob_cont} and the directed last passage percolation result of Theorem~\ref{thm:cont_lpp} are \emph{the same result} viewed two different ways. 
\end{rem}

\begin{proof}[Proof of Theorem~\ref{thm:disc_lpp}]

Let us recall from Theorem~\ref{thm:greene_disc_inf} that, in distribution, we have $L_1^{\rm geo} = L_2^{\rm geo} = \lambda_1$ where $\lambda_1$ is the first part of the principally specialized Schur measure given in~\eqref{eq:schur_inf}. As this measure is explicitly determinantal~\cite{oko01, or03}, it follows that the distribution of $L = \lambda_1$ is the following Fredholm determinant:
\begin{equation}
    \P(L \leq n) = \det(1-\tilde{K}_d)_{\ell^2 \{ n+1, n+2, \dots \}}
\end{equation}
where the kernel $\tilde{K}_d$ is given by 
\begin{equation}
    \tilde{K}_d (k, \ell) = \oint\limits_{|z| = 1 + \delta} \frac{dz}{2 \pi i} \oint\limits_{|w| = 1 - \delta} \frac{dw}{2 \pi i} \frac{F_d(z)}{F_d(w)} \frac{w^{\ell}}{z^{k+1}} \frac{1}{z-w}, \quad F_d (z) = \frac{(\sqrt{a} \tilde{Q}^{1/2}/z; \tilde{Q})_{\infty}}{(\sqrt{a} Q^{1/2}z; Q)_{\infty}}
\end{equation}
with $0 < \delta$ small enough so that $\sqrt{a} \tilde{Q}^{1/2} < 1 - \delta < 1 + \delta < (\sqrt{a} Q^{1/2})^{-1}.$ (In passing, we note $\tilde{K}_d (k, \ell) = K_d(0, k+M; 0, \ell+M)\left|_{M=N=\infty}\right.$ meaning we take $K_d$ at time 0, remove the shift by $M$, and take $M=N=\infty$.)

The asymptotic analysis that follows is almost identical to that of Appendix C.3 of~\cite{bfo20}, except the integrand is different and we use different estimates for the main part of the integrand. Computations are very similar to the ones proving Theorem~\ref{thm:cont_corr}.

We are interested in the following limit (recall $(Q, \tilde{Q}) = (q^\eta, q^\theta)$):
\begin{equation} \label{eq:limit_inf}
    q = e^{-\epsilon}, \quad a = e^{-\alpha \epsilon}, \quad (k, \ell) = \left( \frac{x}{\epsilon} - \frac{\log(\epsilon \eta)}{\epsilon \eta} - \frac{\log(\epsilon \theta)}{\epsilon\theta}, \frac{y}{\epsilon} - \frac{\log(\epsilon \eta)}{\epsilon \eta} - \frac{\log(\epsilon \theta)}{\epsilon \theta} \right), \quad \epsilon \to 0+
\end{equation}
where $ \alpha \geq 0 $ is a priori fixed. 

We can estimate the integrand using the following formula (valid for real $c \notin -\N$):
\begin{equation}
\log (u^c; u)_{\infty} = -\frac{\pi^2}{6} r^{-1} + \left( \frac{1}{2} - c \right) \log r + \frac{1}{2} \log(2 \pi) - \log \Gamma(c) + O(r), \quad u = e^{-r}, \quad r \to 0+
\end{equation}
which\footnote{In words, this means the $q$-Gamma function converges to the Gamma function as $q \to 1-.$} holds uniformly for $c$ in compact sets and we use it for $u \in \{Q, \tilde{Q}\}$. We change variables as $(z, w) = (e^{\epsilon \zeta}, e^{\epsilon \omega})$, the contours becoming vertical lines close and parallel to $i \R$. Then we have the estimate
\begin{equation}
    \frac{F_d(z)}{F_d(w)} \approx \exp\left[ (\omega-\zeta)\left( \frac{\log(\epsilon \eta)}{\eta} + \frac{\log(\epsilon \theta)}{\theta} \right) \right] \frac{\Gamma\left( \frac{\alpha}{2\eta} - \frac{\zeta}{\eta} + \frac{1}{2} \right) \Gamma\left( \frac{\alpha}{2\theta} + \frac{\omega}{\theta} + \frac{1}{2} \right)}{\Gamma\left( \frac{\alpha}{2\theta} + \frac{\zeta}{\theta} + \frac{1}{2} \right) \Gamma\left( \frac{\alpha}{2\eta} - \frac{\omega}{\eta} + \frac{1}{2} \right)}
\end{equation}
and we notice that the $(k, \ell)$ scaling together with the change of variables clear away the $\epsilon$-dependent exponential factor once $w^\ell/z^k$ is plugged in.

By choosing the $(\zeta, \omega)$ contours close to $i\R$ and bounded away from the poles (which are now the arithmetic progressions $(\frac{\alpha}{2} + \frac{(2i-1)\eta}{2})_i, (- \frac{\alpha}{2} - \frac{(2i-1)\theta}{2})_i, \ i \geq 1$ for $\zeta$ and $\omega$ respectively), we have that the integrand, multiplied by $\epsilon^{-1}$, is bounded. It then follows that 
\begin{equation}
    \epsilon^{-1} \tilde{K}_d (k, \ell) \to \tilde{K}_{he}(x, y)   
\end{equation}
uniformly for $x, y$ in compact sets.

To finish, we are left with showing convergence of Fredholm determinants to Fredholm determinants. This will follow from applying the usual Hadamard bound argument once we show exponential decay of the kernel. First denote $\tilde{K}_d^{\epsilon}(x, y) = \epsilon^{-1} \tilde{K}_d(k(\epsilon), \ell(\epsilon))$ where the dependence of $k$ and $\ell$ on $\epsilon$ is given in equation~\eqref{eq:limit_inf}. Next let $\gamma = \min(\alpha/2 + \theta/2, \alpha/2 + \eta/2)/4$ (one fourth the distance from 0 to the nearest pole of the $\Gamma$ functions involved) and let $\nu = 2 \gamma$. The dependence of $\tilde{K}_d^{\epsilon}$ on $x, y$ comes from $k=k(\epsilon)$ and $\ell=\ell(\epsilon)$ and that dependence is simple: $w^\ell/z^k$. Choose contour radii for $z$ and $w$ such that $|z| \geq 1 + (\nu + \gamma) \epsilon$ and $|w| \leq 1 + (\nu-\gamma) \epsilon$ (such contours are possible due to our choice of $\nu, \gamma$ close enough to $0$). We thus have: $|w^\ell/z^k| \leq \frac{  (1+(\nu-\gamma) \epsilon)^{y/\epsilon}   }{(1+(\nu+\gamma) \epsilon)^{x/\epsilon}} \precsim e^{-\nu(x-y)} e^{-\gamma(x+y)}$ and the same $\precsim$ inequality can be written for $|\tilde{K}_d^{\epsilon}|$. This means that the conjugated kernel $\tilde{K}_d^{\epsilon, \rm conj} (x, y) = e^{\nu x}e^{-\nu y} \tilde{K}_d^{\epsilon} (x, y)$ satisfies exponential decay
\begin{equation}
    |\tilde{K}_d^{\epsilon, \rm conj} (x, y)| \precsim e^{- \nu (x+y)}
\end{equation}
and since conjugation does not change Fredholm determinants, we have shown that in the limit~\eqref{eq:limit_inf} and with $n = \frac{s}{\epsilon} - \frac{\log(\epsilon \eta)}{\epsilon \eta} - \frac{\log(\epsilon \theta)}{\epsilon\theta}$ we have
\begin{equation}
    \det(1-\tilde{K}_d)_{\ell^2 \{ n+1, n+2, \dots \}} \to \det(1-\tilde{K}_{he})_{L^2 (s, \infty)}.
\end{equation}
This concludes the proof.
\end{proof}

\begin{rem}
    The distribution of the random variable $L = L_{1}^{\rm geo} = L_{2}^{\rm geo}$ (the latter two have the same distribution) from the above proof is of course by Theorem~\ref{thm:greene_disc_inf} the same as that of $\Lambda_{1,1}$, the first (corner) part of a random plane partition of measure~\eqref{eq:pp_measure} and unrestricted ($M=N=\infty$) base. In short $L = \Lambda_{1,1}$ in distribution. When $\eta = \theta = 1$, this distribution is known to be Gumbel~\cite[Thm.~1]{vy06}. Notice that the Vershik--Yakubovich result is slightly stronger than ours: the estimate is improved by a factor of the form $- \log \log |\epsilon|-\log 2$, meaning they look at the limit $\lim_{\epsilon \to 0} \P(\epsilon L + 2 \log \epsilon - \log |\log \epsilon| - \log 2 < s)$ which they find to be Gumbel via elementary means. 
\end{rem}

\section{Conclusion} \label{sec:conclusion}

In this note we have used the theory of Schur processes to introduce and discuss probabilistic aspects of (integrable) discrete and continuous Muttalib--Borodin processes. We have also shown how simple asymptotic estimates yield the (we believe universal) hard-edge fluctuation behavior of these processes and how this behavior has natural interpretations in terms of discrete and continuous inhomogeneous last passage percolation models.  

Several further questions would be worth investigating. We include a few below.
\begin{itemize}
    \item Both LPP results we have given here, including the continuous one, are in ``zero temperature''. It would be interesting to investigate lifts of these (asymptotic) results to positive temperature and polymer partition functions, in the spirit of~\cite{bcf14}. We note that the combinatorial side is by now well established, see~\cite{kir01, ny04, cosz14, osz14, pm17, boz20}.
    \item For brevity, we have not investigated \emph{all possible asymptotic regimes} for our parameters and we suspect that interesting behavior happens in other cases. One can look at~\cite{joh08} for some clues. Alternatively and intuitively, if $\theta, \eta \to 0$ with $a$ (in the discrete case) and $\alpha$ (in the continuous case) fixed, one should see Tracy--Widom Airy and not Bessel fluctuations (by comparing with~\cite{joh00}), though a little care is needed since the model becomes singular in the $M=N=\infty$ case and presumably one would want $M, N\to \infty$ at comparable rate with $\eta, \theta \to 0$. We plan to address some of these other asymptotic regimes in a future note if they ``turn out to be interesting''.
    \item One can consider more general (and in fact finite) inhomogeneously weighted LPP geometries than the ones considered here; one can also consider other discrete tiling models---the Aztec diamond is an easy example---giving rise to MBPs and their connections to LPP models. Some aspects of this have already been hinted at by Borodin--Gorin--Strahov~\cite{bgs19}, and we hope to address this in some level of generality in future work with E. Strahov. 
    \item We have not addressed at all most of the other asymptotic phenomena one finds in plane partitions: limit shapes, corner (GUE minors) processes, bulk behavior, etc. See~\cite{or03, or06, or07} for a flavor. 
    \item While the connection to Muttalib--Borodin ensembles is mostly lost, considering principally specialized Schur processes of the type we do in Section~\ref{sec:disc_proofs} should yield interesting results if we change boundary conditions and look at pfaffian processes/free boundaries~\cite{bbnv18} or cylindric boundaries/finite temperature fermions~\cite{bb19}.
    \item We have also not addressed any limits of our discrete and continuous processes beyond the Jacobi-like limit we discuss. In particular, one would hope non-trivial behavior can be found even at lower-level limits of Hermite-type perhaps generalizing GUE Dyson Brownian motion. Baryshnikov has already established a link between the latter and LPP models~\cite{bar01}. 
    \item Okounkov~\cite{oko01} (see also references therein for the general context) has shown that correlation functions (and in fact even more general observables) for powersum specialized Schur measures (Schur measures in Miwa variables) satisfy differential equations of integrable type belonging to the 2-Toda hierarchy (and by limits and specializations, the KP and KdV hierarchies). Are there interesting such equations for the principally specialized Schur measures considered here? Furthermore, are there ``nice'' differential or integro-differential equations satisfied by gap probabilities from the various kernels mentioned above (and below)?
\end{itemize}

\appendix

\section{Alternative formulas for our kernels} \label{sec:appendix}

In this section we give alternative formulas for our kernels $K_d, K_c$ and $K_{he}$. Some formulas have different contours of integration, while others are explicit sums of (basic) hypergeometric type.

\subsection{Alternative formulas for $K_d$} \label{sec:alt_d}

The first alternative formula for $K_d$ involves simple if tedious contour manipulations.

\begin{prop} \label{prop:K_d_alt_1}
    Consider the discrete kernel $K_d(s, k; t, \ell)$ for finite integer parameters $N \geq M \geq 1$. We have 
    \begin{equation}
        K_d(s, k; t, \ell) = -\oint_{C_z} \frac{dz}{2 \pi i} \oint_{C_w} \frac{d w}{2 \pi i} \frac{F_d(s, z)}{F_d(t, w)} \frac{w^{\ell-M}}{z^{k-M+1}} \frac{1}{z-w} - V_d(s, k; t, \ell)
    \end{equation}
    where $F_d$ is the same as in Theorem~\ref{thm:disc_corr}; where 
    \begin{equation}
        V_d(s, k; t, \ell) = \Id_{[s > t]} \cdot \begin{dcases}
            \Id_{[s \geq 0]} \oint_{C'_z} \frac{z^{\ell-k-1} dz}{2 \pi i} \cdot (\text{$z$ function on RHS of eq.~\eqref{eq:V_int_d}}), & \text{if } k < \ell, \\
            \Id_{[0 > t]} \oint_{C''_z} \frac{z^{\ell-k-1} dz}{2 \pi i} \cdot (\text{$z$ function on RHS of eq.~\eqref{eq:V_int_d}}), & \text{if } k \geq \ell;
        \end{dcases}
    \end{equation}
    and where the contours are as follows: 
    \begin{itemize}
        \item $C_z$ is a simple closed counter-clockwise contour (possibly disconnected) encircling the finitely many $z$-poles of the integrand of the form $(\sqrt{a} Q^{1/2+\cdots})^{-1}$ and nothing else;
        \item $C_w$ is a simple closed counter-clockwise contour (possibly disconnected) encircling the finitely many $w$-poles of the integrand of the form $\sqrt{a} \tilde{Q}^{1/2+\cdots}$ and nothing else (and in particular not $0$);
        \item $C'_z$ (respectively $C''_z$) is a simple closed counter-clockwise (respectively clockwise) contour (possibly disconnected) encircling the finitely many $z$-poles of the integrand of the form $\sqrt{a} \tilde{Q}^{1/2+\cdots}$ (respectively of the form $(\sqrt{a} Q^{1/2+\cdots})^{-1}$) and nothing else.
    \end{itemize}
\end{prop}

\begin{proof} 
    For the double contour integrals we take the contours given in Theorem~\ref{thm:disc_corr} and first observe that since $0$ is not a pole for $w$, we can just take it out of the contour. For that, the cases $t \leq 0$ and $t > 0$ are considered separately. If $t \leq 0$ we have $\ell - M \geq -L_t$\footnote{Recall the slice at position $t$, $l^{(t)}$, has $l^{(t)}_i \geq M-L_t$, $\forall \ 1 \leq i \leq L_t$, since $l^{(t)}_i = \lambda^{(t)}_i + M - i$ and $\lambda^{(t)}_i \geq 0$.} and we have a degree $N$ monomial $w^N$ coming from the length $N$ $\tilde{Q}$-Pochhammer symbol in the numerator of $F_d$ so overall we have a possible contribution of $w^{N+\ell-M}$ with exponent $N+\ell-M \geq 0$ showing $w=0$ cannot be a pole. For $t>0$ the power of $w$ which matters is $w^{N-t+\ell-M}$ and again the exponent $N-t+\ell-M \geq N-t-L_t \geq 0$ with the last inequality following from the definition of $L_t$ in~\eqref{eq:def_Lt}. Similarly tedious computations show that $\infty$ is not a pole for the $z$-integrand. We can move the initial contour $|z|=1+\delta$ (for some very small $\delta > 0$) through $\infty$ to close it again on the other side encircling only the desired finitely many poles. Reversing the direction to orient the new contour counter-clockwise yields the overall minus sign.

    The same arguments lead to the form of the $V_d$ function. For the case $k \geq \ell$ we see $\infty$ is not a pole for the integrand, and one can deform the appropriate contours via $\infty$ changing their orientation; for the case $k < \ell$, 0 is not a pole and can be excluded from the contour. One notices though that in certain cases the contours contain no poles whatsoever due to the specific form of $F_d(s, z)/F_d(t, z)$ on the right-hand side of equation~\eqref{eq:V_int_d}, and this leads to some additional indicators \emph{as indicated}. 
\end{proof}

There is yet another form of $K_d$ that we give next. Recall the notation of~\eqref{eq:plus_minus_parts}. We then have the following proposition.

\begin{prop} \label{prop:K_d_alt_2}
    It holds that 
\begin{equation}
    \begin{split}
    K_d(s, k; t, \ell) = (\sqrt{a})^{k + \ell - 2M + 2} \sum_{i=0}^{M - s^- - 1} \sum_{j=0}^{N - t^+ - 1}  \left[ (-1)^{i+j}    
    \frac{  Q^{(s^- + i + \tfrac12)(k-M+1)} \tilde{Q}^{(t^+ + j + \tfrac12)(\ell-M+1)}  } {1 - a Q^{s^- + i + \tfrac12}  \tilde{Q}^{t^+ + j + \tfrac12} }  \right. \\
    \left. \times \frac{Q^{\binom{i+1}{2}} \tilde{Q}^{\binom{j+1}{2}} (a Q^{s^- + i + \tfrac12} \tilde{Q}^{s^+ + \tfrac12}; \tilde{Q})_{N-s^+} (a Q^{t^- + \tfrac12} \tilde{Q}^{t^+ + j + \tfrac12}; Q)_{M-t^-} }{(Q; Q)_i (Q; Q)_{M - s^- - 1 - i} (\tilde{Q}; \tilde{Q})_j (\tilde{Q}; \tilde{Q})_{N - t^+ - 1 - j}}\right] - V_d(s, k; t, \ell)
    \end{split}
\end{equation}
where 
\begin{equation}
    V_d(s, k; t, \ell) = \Id_{[s > t]} \cdot \begin{dcases}
        \Id_{[s \geq 0]} \sum_{i=0}^{s^+ - t^+ - 1} \frac{(-1)^i \tilde{Q}^{\binom{i+1}{2}} \left( \sqrt{a} \tilde{Q}^{t^+ + i + \tfrac12} \right)^{k-\ell}  } {  (\tilde{Q};\tilde{Q})_i  (\tilde{Q};\tilde{Q})_{s^+ - t^+ - 1 - i} (a Q^{\tfrac12} \tilde{Q}^{i + \tfrac12} ; Q)_{t^-}  }, & \text{if } k < \ell, \\
        \Id_{[0 > t]} \sum_{i=0}^{t^- - s^- - 1} \frac{(-1)^i Q^{\binom{i+1}{2}} \left( \sqrt{a} Q^{s^- + i + \tfrac12} \right)^{\ell-k}  } {  (Q;Q)_i  (Q;Q)_{t^- - s^- - 1 - i} (a Q^{i + \tfrac12} \tilde{Q}^{\tfrac12} ; \tilde{Q})_{s^+}  }, & \text{if } k \geq \ell.
    \end{dcases}
\end{equation}
\end{prop}

\begin{proof}
    The proof consists of massive residue calculations from the formulas given in Proposition~\ref{prop:K_d_alt_1} bearing in mind that every single pole inside every single contour is simple.
\end{proof}

\subsection{Alternative formulas for $K_c$} \label{sec:alt_c}

We first write $K_c$ in terms of the Fox $H$-function~\eqref{eq:fox_h}. Recall that it depends on integer parameters $N \geq M \geq 1$ and continuous parameters $\alpha \geq 0$ and $\eta, \theta > 0$.

\begin{prop} \label{prop:K_c_alt_0}
    The continuous kernel $K_c$ has the following form:
\begin{equation}
    K_{c} (s, x; t, y) = \frac{1}{\sqrt{xy}} \int_0^1 f_c^{(s)}\left(\frac{1}{ux}\right) g_c^{(t)}(uy) \frac{du}{u} - \Id_{[s > t]} \frac{h(s, x; t, y)}{\sqrt{xy}} 
\end{equation}
where $h(s, x; t, y)$ has the same form as in Proposition~\ref{prop:fox_h} and where
\begin{equation}
    \begin{split}
    f_c^{(s)}(x) &= \begin{dcases}
        \frac{\theta^N}{\eta^{M-|s|}} H^{1, 1}_{2, 2} \left[ x \left| \begin{array}{c} \left(\frac{\alpha}{2 \eta} + |s| + \frac{1}{2} , \frac{1}{\eta}\right), \left(\frac{\alpha}{2 \eta} + M + \frac{1}{2} , \frac{1}{\eta}\right) \\ \left(\frac{\alpha}{2 \theta} + N + \frac{1}{2} , \frac{1}{\theta} \right), \left( \frac{\alpha}{2 \theta} + \frac{1}{2} , \frac{1}{\theta} \right)  \end{array} \right] \right. & \text{if } s \leq 0, \\
        \frac{\theta^{N-s}}{\eta^M} H^{1, 1}_{2, 2} \left[ x \left| \begin{array}{c} \left(\frac{\alpha}{2 \eta} + \frac{1}{2} , \frac{1}{\eta}\right), \left(\frac{\alpha}{2 \eta} + M + \frac{1}{2} , \frac{1}{\eta} \right) \\ \left(\frac{\alpha}{2 \theta} + N + \frac{1}{2} , \frac{1}{\theta} \right), \left(\frac{\alpha}{2 \theta} + s + \frac{1}{2} , \frac{1}{\theta} \right)  \end{array} \right] \right. & \text{if } s > 0,
    \end{dcases} \\
    &\ \\
    g_c^{(s)}(x) &= \begin{dcases}
        \frac{\eta^{M-|s|}}{\theta^N} H^{1, 1}_{2, 2} \left[ x \left| \begin{array}{c} \left(\frac{\alpha}{2 \eta} + M + \frac{1}{2} , \frac{1}{\eta} \right), \left(\frac{\alpha}{2 \eta} + |s| + \frac{1}{2} , \frac{1}{\eta}\right) \\ \left( \frac{\alpha}{2 \theta} + \frac{1}{2} , \frac{1}{\theta} \right), \left(\frac{\alpha}{2 \theta} + N + \frac{1}{2} , \frac{1}{\theta} \right)  \end{array} \right] \right. & \text{if } s \leq 0, \\
        \frac{\eta^M}{\theta^{N-s}} H^{1, 1}_{2, 2} \left[ x \left| \begin{array}{c} \left(\frac{\alpha}{2 \eta} + M + \frac{1}{2} , \frac{1}{\eta} \right), \left(\frac{\alpha}{2 \eta} + \frac{1}{2} , \frac{1}{\eta}\right) \\ \left(\frac{\alpha}{2 \theta} + s + \frac{1}{2} , \frac{1}{\theta} \right), \left(\frac{\alpha}{2 \theta} + N + \frac{1}{2} , \frac{1}{\theta} \right)  \end{array} \right] \right. & \text{if } s > 0.
    \end{dcases}
    \end{split}
\end{equation}
\end{prop}

\begin{proof}
    The proof is the same as that of Proposition~\ref{prop:fox_h} (in the remarks preceding it). Note the difference between the $f_c$ and $g_c$ functions is just the order in which the top and bottom parameters appear (in addition to having ``different times'' in the kernel); this is explained by our integral being of a very special form as in Proposition~\ref{prop:fox_h}.
\end{proof}

Next we have a result parallel to Proposition~\ref{prop:K_d_alt_1}: i.e.~we simply rewrite the original (double) integration contours for $K_c$ as closed contours around only the relevant poles.  

\begin{prop} \label{prop:K_c_alt_1}
    We have 
    \begin{equation}
        K_c(s, x; t, y) = -\frac{1}{\sqrt{xy}} \oint_{C_\zeta} \frac{d \zeta}{2 \pi i} \oint_{C_\omega} \frac{d \omega}{2 \pi i} \frac{F_c(s, \zeta)}{F_c(t, \omega)} \frac{x^\zeta}{y^\omega} \frac{1}{\zeta-\omega} - V_c(s, x; t, y)
    \end{equation}
    where $F_c$ is the same as in Theorem~\ref{thm:cont_corr}; where 
    \begin{equation}
        V_c(s, x; t, y) = \frac{\Id_{[s > t]}} {\sqrt{xy}} \cdot \begin{dcases}
            \Id_{[s \geq 0]} \oint_{C'_\zeta} \frac{  d \zeta}{2 \pi i} \cdot  \left(\frac{x}{y}\right)^\zeta \cdot (\text{$\zeta$ function on RHS of eq.~\eqref{eq:V_int_c}}), & \text{if } x > y, \\
            \Id_{[0 > t]} \oint_{C''_\zeta} \frac{  d \zeta}{2 \pi i} \cdot  \left(\frac{x}{y}\right)^\zeta \cdot (\text{$\zeta$ function on RHS of eq.~\eqref{eq:V_int_c}}), & \text{if } x \leq y;
        \end{dcases}
    \end{equation}
    and where the contours are as follows: 
    \begin{itemize}
        \item $C_\zeta$ is a simple closed counter-clockwise contour (possibly disconnected) encircling the finitely many $\zeta$-poles of the integrand of the form $\frac{\alpha}{2}+\frac{\eta}{2}+\cdots$ and nothing else;
        \item $C_\omega$ is a simple closed counter-clockwise contour (possibly disconnected) encircling the finitely many $\omega$-poles of the integrand of the form $-\frac{\alpha}{2}-\frac{\theta}{2}-\cdots$ and nothing else (and in particular not $0$);
        \item $C'_\zeta$ (respectively $C''_\zeta$) is a simple closed counter-clockwise (respectively clockwise) contour (possibly disconnected) encircling the finitely many $\zeta$-poles of the integrand of the form $-\frac{\alpha}{2}-\frac{\theta}{2}-\cdots$ (respectively of the form $\frac{\alpha}{2}+\frac{\eta}{2}+\cdots$) and nothing else.
    \end{itemize}
\end{prop}

\begin{proof}
    The proof follows by taking the limit $q \to 1-$ in the kernel from Proposition~\ref{prop:K_d_alt_1} with
    \begin{equation}
        q = e^{-\epsilon}, \quad a = q^{-\alpha \epsilon}, \quad (k, \ell) = -\epsilon^{-1} (\log x, \log y), \quad \epsilon \to 0+
    \end{equation}
    and with $Q = q^\eta, \tilde{Q}=q^{\theta}$. All that is used is that $(u^c; u)_n/(1-u)^n \to (c)_n$ as $u \to 1-$; here $u \in \{ Q, \tilde{Q} \}$; and the need for introducing denominators $(1-u)^n$ for appropriate $n$ and $u$ as before explains some of the peculiar factors like the powers of $\eta$ and $\theta$ embedded in $F_c$. In this asymptotic regime we have
    \begin{equation}
        \epsilon^{s-t-1} K_d(s, k; t, \ell) \to K_c(s, x; t, y)
    \end{equation}
    after changing the integration variables on the left to $(z, w) = (1 + \zeta \epsilon, 1 + \omega \epsilon)$ and interchanging limits and integrals. Here we have to do this with care but we note the same arguments of~\cite[Prop.~45]{bfo20} apply mutatis mutandis. Second, the proof of Proposition 45 of op.~cit.~also insures the contours remain closed throughout: we just have to choose them a bounded distance away from the finitely many relevant poles. The limit is uniform for $x, y$ in compact sets. Finally we recall the overall factor $1/\sqrt{xy}$ in front of everything comes from the Jacobian $dk \propto dx/x$ (and $1/x$ becomes $1/\sqrt{xy}$ after conjugation of the kernel).
\end{proof}

Finally, we write $K_c$ as a finite sum (of residues). This will allow us to explicitly compare $K_c$ to Borodin's Muttalib--Jacobi finite-$N$ kernel.

\begin{prop} \label{prop:K_c_alt_2}
    Recalling the notation of~\eqref{eq:plus_minus_parts}, it holds that 
\begin{equation}
    \begin{split}
    K_c(s, x; t, y) = C_{\eta, \theta}^{s, t} \cdot (xy)^{\tfrac{\alpha-1}{2}} \sum_{i=0}^{M - s^- - 1} \sum_{j=0}^{N - t^+ - 1}  \left[ (-1)^{i+j}    
    \frac{  x^{\eta (s^- + i + \tfrac12)} y^{\theta (t^+ + j + \tfrac12)}  } {\alpha + \eta(s^- + i + \tfrac12) + \theta(t^+ + j + \tfrac12) }  \right. \\
    \left. \times \frac{\left(\tfrac{\alpha}{\theta} + (s^- + i + \tfrac12) \tfrac{\eta}{\theta} + (s^+ + \tfrac12) \right)_{N-s^+}  \left(\frac{\alpha}{\eta} + (t^- + \tfrac12) + (t^+ + j + \tfrac12) \tfrac{\theta}{\eta} \right)_{M-t^-}       }{ i! j! (M - s^- - 1 - i)! (N - t^+ - 1 - j)! }\right] - V_c(s, x; t, y)
    \end{split}
\end{equation}
with $C_{\eta, \theta}^{s, t} = \eta^{s^- - t^- + 1} \theta^{t^+ - s^+ + 1}$ and with  
\begin{equation}
    V_c(s, x; t, y) = \frac{\Id_{[s > t]}} {\sqrt{xy}} \cdot \begin{dcases}
        \Id_{[s \geq 0]} \sum_{i=0}^{s^+ - t^+ - 1} \frac{   \eta^{- t^-} \theta^{t^+ - s^+ + 1}  (-1)^i  \left[\frac{x}{y}\right]^{\tfrac{\alpha}{2} + \theta (t^+ + i + \tfrac12)}   } { i!  (s^+ - t^+ - 1 - i)! \left( \tfrac{\alpha}{\eta} + \tfrac12 + (i + \tfrac12) \tfrac{\theta}{\eta} \right)_{t^-} }, & \text{if } x > y, \\
        \Id_{[0 > t]} \sum_{i=0}^{t^- - s^- - 1} \frac{  \eta^{s^- - t^- + 1} \theta^{- s^+}  (-1)^i \left[ \frac{y}{x} \right]^{ \tfrac{\alpha}{2} + \eta(s^- + i + \tfrac12)} }   {  i!  (t^- - s^- - 1 - i)! \left( \tfrac{\alpha}{\theta} + (i + \tfrac12)\tfrac{\eta}{\theta} + \tfrac12 \right)_{s^+}   }, & \text{if } x \leq y.
    \end{dcases}
\end{equation}
\end{prop}

\begin{rem} \label{rem:bor_lim_fin}
    One can now check that in the case $s=t=0, M=N, \eta = 1$ the formula for $K_c$ above matches, \emph{up to conjugation}, that of Borodin's finite kernel~\cite[eq.~(3.5)]{bor99} (the reader should incorporate the factors $x_i^{\alpha_B}$ in Proposition 3.3 of op.~cit.~back into the kernel). We argue as follows: first conjugate $K_{c}$ by $\left(\frac{x}{y}\right)^{\tfrac{\alpha}{2} + \tfrac{\theta}{2} - \tfrac12}$, and then denote $\alpha_B = \alpha + \tfrac{\theta}{2} - \tfrac12$. We then obtain
    \begin{equation}
        K^{\rm conj}_{c}(0, x; 0, y) = \theta x^{\alpha_B} \sum_{i=0}^{N-1} \sum_{j=0}^{N-1} \frac{ (\alpha_B+1+j \theta)_N \left(\frac{\alpha_B+1+i}{\theta}\right)_N} {i!j!(M-1-i)!(N-1-j)!} \frac{x^i y^{j \theta}}{\alpha_B + 1 + i + \theta j}
    \end{equation}
    agreeing with (the aforementioned finite-$N$) Borodin's kernel~\cite{bor99}.
\end{rem}

\begin{proof}
    The proof follows by taking the same limit $q \to 1-$ as in the proof of the previous proposition. Again we only use $(u^c; u)_n/(1-u)^n \to (c)_n$ as $u \to 1-$ with $u \in \{ Q, \tilde{Q} \}$ and the denominators $(1-u)^n$ (for appropriate $n$ and $u$) explain some of the peculiar factors like the powers of $\eta$ and $\theta$. As before
    \begin{equation}
        \epsilon^{s-t-1} K_d(s, k; t, \ell) \to K_c(s, x; t, y)
    \end{equation}
    as $\epsilon \to 0+$ but now we note \emph{there are no convergence issues} as all sums in Proposition~\ref{prop:K_d_alt_2} are finite. Again the overall factor $1/\sqrt{xy}$ comes from the Jacobian $dk \propto dx/x$ and conjugation.

    The same result can be alternatively proven by doing residue calculus on the formulas in Proposition~\ref{prop:K_c_alt_1} noting that every pole in every integral is simple.
\end{proof}

\subsection{Alternative formulas for $K_{he}$} \label{sec:alt_he}

\begin{prop} \label{prop:K_he_alt_2}
    It holds that 
\begin{equation}
        \begin{split}
        & K_{he}(s, x; t, y) = C_{\eta, \theta}^{s, t} \cdot \sum_{i,j=0}^{\infty} \left[    
        \frac{ 1 } {\alpha + \eta(s^- + i + \tfrac12) + \theta(t^+ + j + \tfrac12) }  \right. \\
        & \quad \left. \times \frac{(-1)^{i+j}  x^{\tfrac{\alpha-1}{2} + \eta (s^- + i + \tfrac12)} y^{\tfrac{\alpha-1}{2} + \theta (t^+ + j + \tfrac12)}}{i! j! \Gamma \left( \tfrac{\alpha}{\theta} + (s^- + i + \tfrac12) \tfrac{\eta}{\theta} + (s^+ + \tfrac12) \right) \Gamma \left(\frac{\alpha}{\eta} + (t^- + \tfrac12) + (t^+ + j + \tfrac12) \tfrac{\theta}{\eta} \right) }\right] - V_{he}(s, x; t, y) \\
        & \quad = - V_{he}(s, x; t, y) + \\
        & C_{\eta, \theta}^{s, t} \int_0^1 (ux)^{\tfrac{\alpha-1}{2} + \eta (s^- + \tfrac12)}  (uy)^{\tfrac{\alpha-1}{2} + \theta (t^+ + \tfrac12)} J_{ \tfrac{\alpha}{\theta} + (s^- + \tfrac12) \tfrac{\eta}{\theta} + (s^+ + \tfrac12), \tfrac{\eta}{\theta}} [ (ux)^\eta ] J_{ \frac{\alpha}{\eta} + (t^- + \tfrac12) + (t^+ + \tfrac12) \tfrac{\theta}{\eta} , \tfrac{\theta}{\eta}} [ (uy)^\theta ] du
    \end{split}
\end{equation}
where $C_{\eta, \theta}^{s, t} = \eta^{s^- - t^- + 1} \theta^{t^+ - s^+ + 1}$; where $J_{a, b}(x) = \sum_{k = 0}^\infty \frac{(-x)^k}{k! \Gamma(a+bk)}$ is the same as in~\eqref{eq:J_fn}; and where $V_{he}(s, x; t, y) = V_{c}(s, x; t, y)$ is given in Proposition~\ref{prop:K_c_alt_2} (for which we do not have a formula in terms of the $J$ function).
\end{prop}

\begin{rem}  \label{rem:bor_lim_he}
    Let us take $s = t = 0$; conjugate $K_{he}$ by $\left(\frac{x}{y}\right)^{\tfrac{\alpha}{2} + \tfrac{\theta}{2} - \tfrac12}$; and denote $\alpha_B = \alpha + \tfrac{\theta}{2} - \tfrac12$. We then obtain
    \begin{equation}
        K^{\rm conj}_{he}(0, x; 0, y) = \theta \int_0^1 J_{\frac{\alpha_B+1}{\theta},\frac{1}{\theta}} [ ux ] J_{\alpha_B+1, \theta} [(uy)^\theta] (ux)^{\alpha_B} du = K_B(x, y)
    \end{equation}
    agreeing (as it should) with Borodin's hard-edge kernel~\cite{bor99} defined in~\eqref{eq:bor_ker}.
\end{rem}

\begin{proof}
    The equivalence between the two forms of the kernel as given is immediate once one uses $A^{-1} = \int_0^1 u^{A-1} du$ with $A = \alpha + \eta(s^- + i + \tfrac12) + \theta(t^+ + j + \tfrac12)$ and then uses the definition of the $J$ function.

    To prove that $K_{he}$ is given by the first equality, we take the form of $K_c$ in~\ref{prop:K_c_alt_2} and take the $M, N\to \infty$ limit 
    \begin{equation}
        \lim_{M, N \to \infty} \frac{1}{M^{\frac{1}{\eta}}N^{\frac{1}{\theta}}} K_c \left( s, \frac{x}{M^{\frac{1}{\eta}}N^{\frac{1}{\theta}}}; t, \frac{y}{M^{\frac{1}{\eta}}N^{\frac{1}{\theta}}} \right) = K_{he} (s, x; t, y)
    \end{equation}
    using \emph{the same standard estimate and tail-pruning argument} as in~\cite[Thm.~3.4 and Sec.~6]{bor99}. We only sketch the argument here. First noticing that, after plugging $(x)_n = \Gamma(x+n)/\Gamma(x)$ into $K_c$ to replace Pochhammer symbols by ratios of $\Gamma$ functions, we have the following $M$ and $N$ dependence in the kernel:
    \begin{equation}
        \begin{split}
            M \text{ dependence } &= \frac{\Gamma(\frac{\alpha}{\eta} + (t^+ + j + \frac12) \frac{\theta}{\eta} + M + \frac12)}{\Gamma(M - s^- - j)}, \\
            N \text{ dependence } &= \frac{\Gamma(\frac{\alpha}{\theta} + (s^- + i + \frac12) \frac{\eta}{\theta} + N + \frac12)}{\Gamma(N - t^+ - j)}.
        \end{split}
    \end{equation}
    One then uses the estimate $\Gamma(z+a) / \Gamma(z+b) \approx z^{a-b}, z \to \infty$ (uniform in compact sets over the other parameters)
    to isolate the asymptotic contribution of the above terms as follows:
    \begin{equation}
        \begin{split}
            M \text{ dependence } &\approx M^{\tfrac{\alpha}{\eta} + (t^+ + j + \tfrac12) \tfrac{\theta}{\eta} + \tfrac12 + (s^- + i)}, \quad M \to \infty, \\
            N \text{ dependence } &\approx N^{\tfrac{\alpha}{\theta} + (s^- + i + \tfrac12) \tfrac{\eta}{\theta} + \tfrac12 + (t^+ + j)}, \quad N \to \infty
        \end{split}
    \end{equation}
    and this concludes the proof after realizing the scaling of $x, y$ and the pre-scaling of $K_c$ cancel these powers.
\end{proof}

\bibliographystyle{myhalpha}
\bibliography{mbe}

\end{document}